\documentclass[11pt, reqno]{amsart}
\allowdisplaybreaks[1]

\usepackage{relsize}
\usepackage{amsmath,amsfonts,amssymb,amsthm,mathrsfs}
\usepackage{hyperref}
\usepackage{cleveref}
\usepackage[margin=2.5cm]{geometry}
\usepackage{setspace}
\usepackage{mathtools}

\usepackage{enumitem}

\usepackage{setspace}
\usepackage{xcolor}


\allowdisplaybreaks[1]

\newcommand{\bh}{\mathcal{B}(\mathcal{H})}
\newcommand{\bk}{\mathcal{B}(\mathcal{K})}

\newcommand{\cx}{C(X)}

\newcommand{\ox}{\mathcal{O}(X)}

\newcommand{\be}{\begin{equation}}
	\newcommand{\ee}{\end{equation}}
\newcommand{\bea}{\begin{eqnarray}}
	\newcommand{\eea}{\end{eqnarray}}
\newcommand{\bean}{\begin{eqnarray*}}
	\newcommand{\eean}{\end{eqnarray*}}
\newcommand{\brray}{\begin{array}}
	\newcommand{\erray}{\end{array}}

\newtheorem{dfn}{Definition}[section]
\newtheorem{thm}[dfn]{Theorem}
\newtheorem{lmma}[dfn]{Lemma}
\newtheorem{ppsn}[dfn]{Proposition}
\newtheorem{crlre}[dfn]{Corollary}
\newtheorem{xmpl}[dfn]{Example}
\newtheorem{rmrk}[dfn]{Remark}

\theoremstyle{definition}\newtheorem*{notation*}{Notation}

\newcommand{\cst}{C^*}
\newcommand{\sqd}{D^{1/2}}

\newcommand{\bdfn}{\begin{dfn}\rm}
	\newcommand{\bthm}{\begin{thm}}
		\newcommand{\blmma}{\begin{lmma}}
			\newcommand{\bppsn}{\begin{ppsn}}
				\newcommand{\bcrlre}{\begin{crlre}}
					\newcommand{\bxmpl}{\begin{xmpl}}
						\newcommand{\brmrk}{\begin{rmrk}\rm}
							
							\newcommand{\edfn}{\end{dfn}}
						\newcommand{\ethm}{\end{thm}}
					\newcommand{\elmma}{\end{lmma}}
				\newcommand{\eppsn}{\end{ppsn}}
			\newcommand{\ecrlre}{\end{crlre}}
		\newcommand{\exmpl}{\end{xmpl}}
	\newcommand{\ermrk}{\end{rmrk}}

\newcommand{\bbc}{\mathbb{C}}

\newcommand{\scre}{\mathscr{E}}

\newcommand{\cla}{\mathcal{A}}
\newcommand{\clb}{\mathcal{B}}

\newcommand{\clh}{\mathcal{H}}
\newcommand{\cli}{\mathcal{I}}
\newcommand{\clj}{\mathcal{J}}
\newcommand{\clk}{\mathcal{K}}
\newcommand{\cll}{\mathcal{L}}
\newcommand{\clm}{\mathcal{M}}
\newcommand{\cln}{\mathcal{N}}

\newcommand{\clo}{\mathcal{O}}

\makeatletter
\let\@wraptoccontribs\wraptoccontribs
\makeatother

\title[CP instruments]{Understanding Quantum Instruments Through the Analysis of $C^*$-Convexity and Their Marginals}
	\author{B. V. Rajarama Bhat}
\address{Indian Statistical Institute, Stat Math Unit, R V College Post, Bengaluru, 560059, India.}
\email{bvrajaramabhat@gmail.com, bhat@isibang.ac.in}

\author{Arghya Chongdar}
\address{Indian Statistical Institute, Stat Math Unit, R V College Post, Bengaluru, 560059, India.}
\email{chongdararghya@gmail.com}

\author{Sruthymurali}
\address{Department of Mathematical Sciences,
Kannur University, Mangattuparamba Campus, Kannur, Kerala, 670567, India.}
\email{sruthy92smk@gmail.com,  sruthymurali@kannuruniv.ac.in}
\date{}
\begin{document}
 \begin{abstract}
	Quantum instruments are mathematical devices introduced to describe the conditional state change during a quantum process. They are completely positive map valued measures on measurable spaces.  We may also view them  as non-commutative analogues of joint probability measures.   We analyze the $C^*$-convexity structure of spaces of quantum instruments.  A complete description of the $C^*$-extreme instruments in finite dimensions has been established. Further,  the implications of $C^*$-extremity  between quantum instruments and their marginals has been explored.
\end{abstract}

	\maketitle

	\section{Introduction}\label{Introduction}

 Davies and Lewis ~\cite{davies} introduced a  mathematical framework for quantum measurements.
 Their  key idea was to mathematically define an {\em ``instrument''}  to capture the complexities of quantum measurement processes. According to them it is a concept which generalises that of an observable and that of an operation. It is  particularly  relevant for situations involving continuous observables and repeated measurements. It could be interpreted in different ways \cite{Srinivas}.  A decisive advance in this theory was achieved by Ozawa~\cite{ozawa}, who says that the state changes determined by measuring processes naturally correspond to completely positive instruments and vice versa. In simple terms, quantum instruments or CP instruments are nothing but completely positive (CP) map valued measures. So they can be thought of as generalizations of both positive operator valued measures (POVMs) and CP maps.  From mathematical point of view, Ozawa  developed a dilation theory for CP instruments which became a cornerstone for subsequent progress. In particular, Holevo~\cite{holevo} established a Radon--Nikodym type theorem for CP instruments, offering deeper structural insights. This line of investigation was extended in~\cite{Chiribella_and_Ariano}. More recently, the marginal problems for CP instruments have attracted significant attention, as studied in~\cite{Ariano} and by Juha-Pekka Pellonp\"a\"a  et al  \cite{pellonapaa1, pellonapaa2, pellonpaapieces, Heinosaari_Teiko_and_Pellonpa}. 
 Though we use a somewhat different set-up and do not use the technical framework of direct integrals of Hilbert spaces, many of the results here are motivated by these articles.

 It is easy to see that the space of CP instruments on a measure space taking values in a fixed algebra form a convex set. In one way or the other most of these papers analyze this convexity.   The main goal of this article is to study the role of quantum convexity or $C^*$-convexity (See Definition \ref{cstarconvexity}) in understanding quantum instruments.  
 
	Here is another way of looking at the concept of quantum instruments. Building on the Riesz representation theorem on compact Hausdorff spaces, quantum states can be seen as non-commutative analogues of classical probability measures, while completely positive maps on $\cst$-algebras serve as generalized states. In this analogy, positive operator-valued measures (POVMs)—the quantum counterparts of classical measurements—are regarded as non-commutative generalizations of positive measures.
Following this line of reasoning, the CP instruments introduced by Ozawa can be viewed as a quantization of classical joint probability distributions, where the marginals consist of a CP map on a $\cst$-algebra and a POVM on a measurable space. However, one may also consider alternative quantizations of classical joint measures, where both marginals are either CP maps on $\cst$-algebras or POVMs on measurable spaces. From the perspective of quantum measurement theory, the concept of CP instruments with one marginal being a POVM (which could be called as semi-classical) and another being a CP map (purely quantum) has physical relevance. Mathematically, this structure is interesting as we get to see different features in two marginals.
		
Classically, a joint probability measure on a product space is an extreme point in the convex set of all joint measures if and only if it is the product of its marginals and each marginal is itself extreme. In other words, the extremality of the joint distribution is equivalent to the extremality of the marginals together with the product structure. In particular, the extremality of even one marginal suffices to uniquely determine the entire joint distribution. Thus, in the classical setting, the convex structure of the marginals plays a fundamental role in characterizing the joint distribution.

From a mathematical standpoint, it becomes natural to investigate the structure of CP instruments through their marginals, particularly through their convexity theory.	
In the setting of completely positive (CP) instruments, the classical characterization of extremality holds only in one direction: if both marginals of a CP instrument are extreme, then the instrument itself is extreme. This result was established by Haapasalo, Heinosaari, and Pellonpää (see Theorem 4.1 in \cite{pellonpaapieces}). Furthermore, they established that the extremality of even one of the marginals is sufficient to uniquely determine the instrument (Theorem 4.1, \cite{pellonpaapieces}). However, a product representation of an instrument in terms of its extreme marginals generally does not hold. These findings provide a partial analogue of the classical case in the quantum setting, while simultaneously revealing key structural differences arising from the inherent noncommutativity of quantum observables.
 However, the converse does not hold—that is, a CP instrument may be extreme even if its marginals are not extreme. A concrete example illustrating this phenomenon is discussed in Example \ref{exmp: extreme instrument with non-extreme marginals}.

These observations naturally motivate a re-investigation of such results through the lens of C$^*$-convexity. In this work, we have characterized the set of $C^*$-extreme points within the space of unital completely positive (CP) instruments. Furthermore, we have obtained a partial counterpart of the classical result: if the marginals of a CP instrument are $C^*$-extreme, then the instrument itself is decomposable (see Definition \ref{Decomposableinstruments})—that is, it coincides with the product of its marginals. In particular, the instrument is uniquely determined by its $C^*$-extreme marginals. However, the converse does not hold in general: the $C^*$-extremality of an instrument does not necessarily imply the $C^*$-extremality of its marginals, as demonstrated by Example \ref{exmp : $\cst$-extreme instrument wth non extreme CP marginal}.

The paper is organized as follows. In Section \ref{Basic properties of instruments}, we begin with the definition of instruments on measurable spaces and recall several foundational results, including the bi-dilation theorem and the Radon–Nikodym-type theorem.  We introduce the notion of pure instruments (Definition \ref{Pure instrument}) and characterize them via minimal bi-dilations. A brief discussion is included on sub-minimal dilations and their relation to the bi-dilation associated with a given instrument. We also revisit the concept of decomposable instruments (Definition \ref{Decomposableinstruments}) and provide a characterization in terms of minimal bi-dilation (Theorem \ref{decomposibleequivalent-1}).

In Section \ref{Extremality and $\cst$-Convexity of Instruments}, we establish a Choi-type characterization of extreme instruments in the finite dimensional setting (Corollary \ref{Choi-Kraus characterization of extreme instruments}). We also provide a characterization of extreme instruments using a natural partial order (Theorem \ref{thm: extreme point condition by Bhat}). Our main result—Theorem \ref{characterization of $\cst$-extreme instruments}—gives a structural characterization of $C^*$-extreme unital CP instruments on finite-dimensional Hilbert spaces. 

Next in Section \ref{Instrument and Its marginals through different notions of convexity}, we examine the relationship between an instrument and its marginals by analyzing classical and $C^*$-convexity. We recall the known facts for classical convexity and then derive some results for $C^*$-convexity.  We show that for both classically extreme and $C^*$-extreme instruments with commutative range, the POVM marginal is necessarily spectral. But in general, the two situations
are quite distinct. 
Our final result (Theorem \ref{Final}) tells us that
 the $C^*$-extremality of the marginals implies the $C^*$-extremality of the instrument and something more. 

    For a brief overview of completely positive maps and positive operator-valued measures, we refer the reader to \cite{arvesonsubalgebra, pisiernormal, daviesquantum} and the references therein. Throughout this paper, we use the abbreviations UCP (unital completely positive maps) and POVM (positive operator-valued measures). All our Hilbert spaces are complex with inner products anti-linear in the first variable. For any Hilbert space $\mathcal{H}$, $\bh $ denotes the algebra of all bounded operators on $\mathcal{H}$ and $I_{\clh }$ denotes the identity operator on $\clh .$  

	\section{Basic properties of instruments}\label{Basic properties of instruments}
	
	\subsection{Quantum instruments}
	Here, we review the definition of CP instruments within the framework of general $C^*$-algebras and note some of  their fundamental properties. See \cite{davies}, \cite{holevo}, \cite{ozawa}, \cite{pellonapaa1}, \cite{pellonapaa2}, and \cite{Srinivas} for general literature on the topic of quantum instruments.

    In the following unless otherwise stated, $X$ is a non-empty set and $\mathcal{O}(X)$ denotes a $\sigma$-algebra of subsets of $X.$ The pair $(X,\mathcal{O}(X))$ is called a {\em measurable space.} Whenever $X$ is a topological space, we will consider $\mathcal{O}(X)$ to be the Borel $\sigma$-algebra of $X.$ The topological spaces that are being considered are all Hausdorff. Suppose $\cla $ is a $C^*$-algebra and $\clh $ is a Hilbert space. Then the  set of all completely positive (CP) maps from $\cla$ to the algebra $\bh $ of all bounded operators on $\clh$ is denoted by $CP(\cla,\bh)$.
	To date, the existing theory has focused on CP instruments in von Neumann algebra setting, where the value space consists of normal CP maps. Here, perhaps for the first time, we are considering CP instruments in the broader context of general $C^*$-algebras. A suggestion that for physical reasons we should be looking at finitely additive measures, which amounts going beyond normal maps, was made already in \cite{Srinivas}.  Let us begin with some definitions.

	\begin{dfn}[Quantum instruments]\label{maindefinition}
		Let $(X,\ox)$ be a measurable space. Suppose $\mathcal{A}$ is a unital $C^*$-algebra and $\mathcal{H}$ is a Hilbert space. A  \emph{CP instrument} or \emph{quantum instrument} on $(X,\clo(X))$ with values in $CP(\cla,\bh)$ is a map $\cli: \ox\to CP(\cla,\bh)$ satisfying the following,  
		\begin{enumerate}
			\item $\cli(A)\in CP(\cla,\bh)~,\forall~A\in~\mathcal{O}(X)$;
			\item  For every $h,k\in\mathcal{H}$ and any $a\in\mathcal{A}$, the map $\cli_{a,h,k}:\mathcal{O}(X)\to\mathbb{C}$ defined  by
			\begin{equation}
				\cli_{a,h,k}(A)=\langle h,\cli(A)(a)k\rangle~~\text{for all } A\in\ox,\label{eq:notation for cli_a,h,k}
			\end{equation}
			is a complex measure.
		\end{enumerate}
		
		Moreover, a quantum instrument $\cli$ is said to be 
		\begin{itemize}
   			\item \emph{normalized} or \emph{unital} if $\cli(X)(1_\cla)=1_\mathcal{H}$ i.e. if the completely positive map $\cli(X)$ is \emph{unital}.
			\item \emph{normal} if $\cla$ is a von Neumann algebra and $\cli(A)$ is a  normal CP map for all $A\in\ox.$
			\item \emph{spectral} if $\cli(A)$ is a $*$-homomorphism for all $A\in \ox$ and $\cli $ is normalized.
		\end{itemize}
	\end{dfn}
	
	It follows from the definition of CP instruments that, for any increasing (or
	decreasing)  sequence $\{A_n\}$ of measurable  subsets converging to
	$A$ i.e. $A_n\subseteq A_{n+1}$ and $\cup_nA_n=A$ (or $A_n\supseteq
	A_{n+1}$ and  $\cap A_n=A$), $\cli(A_n)(a)\to \cli(A)(a)$ in weak operator
	topology (WOT) in $\bh,$ for all $a\in \mathcal{A}$. A bounded net, $\{\phi_i: \cla \to \bh\}_{i\in\cli}$ 
	of completely positive maps converges to a completely positive map, $\phi: \cla \to \bh$ in the \emph{bounded weak} (BW) topology if $\phi_i(a) \to \phi(a)~\text{in WOT},~\forall~a~\in \cla.$ Therefore, we infer that in the countable additivity of CP instruments:
	\begin{align*}
		\cli \left(\bigcup_{n=1}^{\infty} B_n\right) =\sum_{n=1}^\infty\cli (B_n),~~B_n\in
		\ox, B_n\cap B_m=\emptyset ~\mbox{for} ~ n\neq m,
	\end{align*} the convergence
	of the series of CP maps holds in BW topology.
	
    The following example shows that CP maps and POVMs can be regarded as special types of instruments.
    \begin{xmpl} \begin{enumerate} \item Let $\mu:\ox\to\bh$ be a POVM. Consider $\cla=\bbc.$ Define a map $\cli_{\mu}:\ox\to CP(\cla,\bh)$ by, $$\cli_{\mu}(A)(a)=a\mu(A),  ~\forall A\in\ox,~a\in\bbc.$$ This defines an instrument, showing that every POVM can be considered as a CP instrument.
			\item Similarly given a CP map $\phi:\cla\to\bh,$ consider the trivial measurable space $\ox=\{\emptyset,X\}$ for some non-empty set $X.$ Define the map $\cli_\phi:\ox\to CP(\cla,\bh)$ by $$\cli (\emptyset  )(a)=0, ~~ \cli_\phi(X)(a)=\phi(a),\forall~a\in\cla.$$ This shows that every CP map can be regarded as a CP instrument in  a natural way.
		\end{enumerate}

	\end{xmpl}
	\begin{rmrk}\label{complex measures associated to an instrument}
		For any instrument $\cli $, by $\cli _{a,h,k}$ we would mean the complex measure
		defined in \eqref{eq:notation for cli_a,h,k}. It is clear that a CP instrument
		$\cli $ is determined by its associated family of complex measures $\{ \cli
		_{a,h,k}: a\in\cla,h,k\in \clh \}.$
	\end{rmrk}

	\begin{rmrk}[Bivariate realization]\label{bivariate-realization}
		Observe that a CP instrument $\cli$ on $(X,\ox)$ with values in $CP(\cla,\bh)$ can be thought of as a bivariate map $\tilde{\cli}: \ox\times \cla \to \bh$ given by,
		\begin{equation*}
			\tilde{\cli}(A,a)=\cli(A)(a),~\forall a \in \cla, ~A \in \ox,
		\end{equation*}
		where, \begin{itemize}
			\item For each $A\in\ox,$ the map, $\tilde{\cli}(A,\cdot):\cla\to\bh$ defined by $\cli(A,a)=\cli(A)(a)$ gives a completely positive map;
			\item Fixing any positive element $a\in\cla,$ the map, $\tilde{\cli}(\cdot,a):\ox\to\bh$ given by $\cli(A,a)=\cli(A)(a)$ defines a positive operator valued measure on $X.$
		\end{itemize} 
		
		Conversely, if there exists a bivariate map $\tilde{\cli}:\ox\times\cla\to \bh$ such that: 
		\begin{enumerate}
			\item for each $A\in\ox,$ the map, $\tilde{\cli}(A,\cdot):\cla\to\bh$ is a completely positive map,
			\item for every $a\in\cla_+,(\cla_+,$ is the set of all positive elements in $\cla)$ the map, $\tilde{\cli}(\cdot,a):\ox\to\bh$ defines a positive operator valued measure on $X,$
		\end{enumerate}
		then we can construct a map $\cli:\ox\to CP(\cla,\bh)$ such that $\cli(A)(a)=\tilde{\cli}(A,a)$ for all $A\in\ox$ and $a\in\cla$ respectively.
		The fact that $\cli$ is a CP instrument follows directly from the definition of $\tilde{\cli}.$
	\end{rmrk}
	Based on Remark \ref{bivariate-realization}, CP instruments can be described unambiguously either in terms of the Definition \ref{maindefinition} or equivalently as a bivariate map, as outlined in Remark \ref{bivariate-realization}.
	\begin{rmrk}
		Bivariate realization of CP instruments implies that corresponding to every CP instrument $\cli:\clo(X)\to CP(\cla,\bh),$ we have a POVM, $\cli(\cdot,1_\cla): \clo(X)\to \bh$ and a CP map, $\cli(X,\cdot):\cla\to \bh$. Furthermore, if the quantum instrument is unital then $\cli(.,1_\cla)$ is a normalized POVM, and $\cli(X,.),$ is an UCP map. These are called respectively as {\em associated\/}  POVM and CP map of the CP instrument $\cli.$ 
	\end{rmrk}

	\begin{notation*}
		We use the symbols $\mu_\cli,\phi_\cli$ to denote the associated POVM and CP map of a quantum instrument $\cli.$ We also call them POVM marginal and CP marginal of the instrument $\cli,$ and collectively as \emph{marginals}.  We may write $\cli(A,a)$ to mean $\cli(A)(a).$  
		Let $Ins_{\clh}(X,\cla)$ denote the
		collection of all CP instruments on $\ox$ with values in $CP(\cla,\bh)$ and let
		$I_{\clh}(X,\cla)$ denote the collection of all normalized elements in
		$Ins_{\clh}(X,\cla).$
	\end{notation*}
	
	Analogous to classical joint measures, in the non-commutative setting, the entire instrument vanishes if and only if either of its marginals vanishes. The following result is elementary but will be used multiple times.
	
	\begin{ppsn}\label{zero set of marginals and the whole are same}
		Let, $\cli:\ox\to CP(\cla,\bh)$ ba a CP instrument. Then,  the following are equivalent:  (i) $\cli=0;$
			(ii)  $\mu_{\cli}=0;$  
			(iii) $\phi_{\cli}=0.$
		
	\end{ppsn}
	\begin{proof}
		If the instrument $\cli = 0$, then by definition, its marginals satisfy $\phi_\cli = 0$ and $\mu_\cli = 0.$
		Conversely, if $\mu_{\cli}=0$ i.e. $\cli(A,1_{\cla})=0,$ for any $A\in\ox,$ in particular, $\cli(X,1_{\cla})=0,$ which implies that $\cli(A)=0,$ for all $A\in\ox.$
		Similarly, we can verify that if the CP marginal $\phi_{\cli}=0$ then the entire instrument $\cli=0.$
	\end{proof}

	\subsection{Bi-dilation theorem}
	The classical dilation theorems of Naimark \cite{neumark} and Stinespring \cite{stinespring} are foundational results in the theory of operator algebras, establishing that POVMs and completely positive (CP) maps can be dilated to spectral measures and $*$-homomorphisms, respectively. Ozawa  (\cite{ozawa}, Proposition 4.1)  proved a dilation theorem for CP instruments, demonstrating that such instruments can be dilated to spectral instruments (as characterized in \ref{characterization of spectral}). Ozawa called POVMs as semiobservables and projection valued measures as observables. Although his treatment focused on normal CP instruments on von Neumann algebras, the underlying arguments generalize seamlessly to quantum instruments on arbitrary $\cst$-algebras. Here, we present this dilation theorem in the general $\cst$-algebra setting along with a sketch of  proof.

	\begin{thm}\label{thm : Bi-dilation theorem} (Bi-dilation Theorem)
		Let $\cli: \ox \to CP(\cla,\bh)$ be an instrument. Then there exists a quadruple $(\clk,\pi,E,V),$ where $\clk$ is a Hilbert space, $\pi: \cla \to \bk$ is a unital $*$-homomorphism, $E: \ox \to \bk$ is a spectral measure, and $V\in\clb(\clh,\clk)$ such that, \begin{equation}\label{naimarkdilation}
			\cli(A)(a)=V^*\pi(a)E(A)V\quad\text{and}\quad E(A)\pi(a)=\pi(a)E(A),
			\forall a \in \cla, A \in \ox
		\end{equation}
		and satisfies the minimality condition: $[\pi(\cla)E(\ox)V\clh]=\clk.$ Such a quadruple is unique up to unitary equivalence.
	\end{thm}   We call the quadruple $(\clk,\pi,E,V),$ the {\em minimal bi-dilation quadruple\/}  for the instrument $\cli.$ Since $\pi$ is a unital $*$-homomorphism and $E$ is spectral, the equation \ref{naimarkdilation} implies that $V$ is an isometry if and only if $\cli$ is a \emph{normalized} instrument. 
	
	The proof of this theorem primarily involves nothing more than the standard GNS construction with some bookkeeping. However, for clarity, we provide a brief outline of the proof.
	
	\begin{proof}
		Consider the product space, $\clm=\ox\times \cla\times\clh$ and define the map $K:\clm\times\clm\to\bbc$ by $K((A,a,h),(B,b,k))=\langle h,\cli(A\cap B)(a^*b)k \rangle.$ For arbitrary $(A_1,a_1,h_1),\cdots,(A_n,a_n,h_n)\in\clm$ and $c_1,\cdots,c_n\in\bbc,$
		\begin{equation*}
			\begin{split}
				\sum_{i,j} \overline{c_i}c_jK((A_i,a_i,h_i),(A_j,a_j,h_j))&=\sum_{i,j}\overline{c_i}c_j \langle h_i,\cli(A_i\cap A_j)(a_i^*a_j)h_j\rangle\\
				&=\langle  (\oplus_{i=1}^n c_ih_i),  [\cli(A_i\cap A_j)(a_i^*a_j)] (\oplus_{j=1}^n c_jh_j)  \rangle
			\end{split}
		\end{equation*}
		where the last inner product is in the direct sum of n-copies of the Hilbert space $\clh.$ Using standard measure theoretic arguments, note that each element from the collection $\{A_1,\cdots,A_n\}$ can be written as a disjoint union of a new collection of disjoint measurable subsets $\{B_1,\cdots,B_k\}.$ Consequently, 
		\begin{equation*}
			\sum_{i,j}\overline{c_i}c_j \langle h_i,\cli(A_i\cap A_j)(a_i^*a_j)h_j\rangle=\sum_{l}\sum_{i,j\in \{r:B_l \subseteq A_r\}}\overline{c_i}c_j \langle h_i,\cli(B_l)(a_i^*a_j)h_j\rangle.
		\end{equation*}
		Due to the complete positivity of the CP instrument $\cli$, the second summand in the above equation is positive. Hence $K$ is a positive definite kernel. By GNS construction we have a Hilbert space $\clk$ with a map $\lambda:\clm\to\clk$ such that, 
		\begin{equation*}
			\langle \lambda(A,a,h),\lambda(B,b,k)\rangle=\langle h, \cli(A\cap B)(a^*b)k\rangle, ~for~ (A,a,h),(B,b,k)\in\clm,
		\end{equation*}
		and the set $\{\lambda(A,a,h):(A,a,h)\in\clm\}$ forms a total set in $\clk.$ For each unitary $u\in\cla,$ we define $\pi(u)$ by $\pi(u)(\lambda(A,a,h))=\lambda(A,ua,h).$ Then it's easy to verify that  $\pi(u)$ is a unitary operator on $\clk.$ By extending $\pi$ linearly via the functional calculus, we obtain a unital $*$-homomorphism from $\cla$ to $\bk.$
		Similarly, define the map, $E:\ox\to\bk$ by $E(B)(\lambda(A,a,h))=\lambda(A\cap B,a,h),$ for each $B\in\ox.$ One can verify that $E$ is a spectral measure on the measurable space $(X,\ox).$ It follows directly from the definitions of the maps $\pi,E$ that $\pi(a)E(A)=E(A)\pi(a)$ for each $a\in\cla$ and $A\in\ox.$ Finally, define a linear map $V:\clh\to\clk$ by $Vh=\lambda(X,1_{\cla},h).$ Routine verification shows that $V$ is bounded, in fact $\|V\|=\|\cli(X)(1_{\cla})\|^{\frac{1}{2}},$ and the desired dilation is given by $\cli(A)(a)=V^*E(A)\pi(a)V.$ The uniqueness follows by standard methods.
	\end{proof}
	
	As a consequence of Theorem \ref{thm : Bi-dilation theorem}, every CP instrument on a finite set, taking values in CP maps on matrix algebras, admits a Choi–Kraus representation analogous to that of CP maps on matrix algebras. Further, we can characterize minimal {\em Choi-Kraus} decompositions in terms of linear independence. In the following, for notational simplicity,  for any singleton $\{i\}$, we denote $\cli(\{i\}, \cdot )$ by $\cli (i, \cdot ).$

	\begin{ppsn}\label{Choi-Kraus}
		Let $\cli:\ox\to CP(M_d, M_k)$ be a CP instrument where $X=\{1,\cdots,n\}$ is   a finite set.   Then there exists a collection of matrices $\{ V^i_{j}\}_{j=1}^{l_i} \subset M_{k \times d}$ for some $l_i\in \mathbb{N}$, for each $i \in \{1, \ldots, n\}$ such that,
    \begin{equation}\label{choi-Kraus}\cli(i,a)=\sum _j{V^i_{j}}^*aV^i_{j},~\forall~a \in M_d.\end{equation}

Moreover,  $\{V^i_j\}_{j=1}^{l_i}$ can be chosen to be linearly independent for each $i\in\{1,\cdots,n\}$.
	\end{ppsn}
	\begin{rmrk}
We call the representation in \ref{choi-Kraus} as the Choi-Kraus decomposition. We refer  the  collection of matrices $\{V^i_j\}_{j=1}^{l_i}$ as the {\em Choi-Kraus} operators associated to the instrument $\cli.$ If they are linearly independent, the decomposition is said to be minimal and then  for any other decomposition $$\cli(i,a)=\sum_k {W^i}^*_{k}aW^i_{k},~~\forall ~a\in M_d,$$
		there exist a family of isometries $(\mu^i_{k,j})$ such that, $$W^i_{k}=\sum_j \mu^i_{k,j}V^i_{j}.$$
	\end{rmrk}

	Another important implication of the dilation theorem is that spectral instruments admit a canonical factorization as the composition of a spectral measure and a unital $*$-homomorphism.

	\begin{thm}\label{characterization of spectral}
		Let $\cli:\ox\to CP(\cla,\bh)$ be a UCP instrument. Then TFAE:
		\begin{enumerate}[label=(\roman*)]
			\item \label{usual representation}$\cli$ is a spectral instrument,
			\item \label{product representation}$\cli(A,a)=\phi_\cli(a)\mu_\cli(A)=\mu_\cli(A)\phi_\cli(a)$ for all $a\in\cla,$~ $A\in\ox,$ where $\phi_\cli:\cla\to\bh$ is a unital $*$-homomorphism and $\mu_\cli:\ox\to\bh$ is a spectral measure.
		\end{enumerate}
        \end{thm}
		\begin{proof}
			The proof,$\ref{product representation}\implies \ref{usual representation}$ is obvious.
			For the converse, let $(\clk,\pi,E,V)$ be the minimal bi-dilation tuple of $\cli.$ As $\cli$ is unital, it follows that $V$ is an isometry. Since, $\cli$ is spectral, both the associated POVM, $\mu_{\cli}$ and the UCP map, $\phi_{\cli}$ are spectral, with $\mu_{\cli}$ being a spectral measure and $\phi_{\cli}$ a $*$-homomorphism respectively. Then for any $A\in\ox$ we have,  
			\begin{equation*}
				\begin{split}
					[V\mu_{\cli}(A)-E(A)V]^*\cdot[V\mu_{\cli}(A)-E(A)V]
					&= [\mu_{\cli}(A)V^*-V^*E(A)]\cdot[V\mu_{\cli}(A)-E(A)V] \\
					& =\mu_{\cli}(A)^2-\mu_{\cli}(A)^2-\mu_{\cli}(A)^2+\mu_{\cli}(A)\\
					&    = \mu_{\cli}(A)^2-\mu_{\cli}(A)\\
                   & =0.
				\end{split}
			\end{equation*}
			So we obtain $V\mu_{\cli}(A)=E(A)V.$ As a result, for all $a\in\cla$ and $A\in\ox,$ we have $$\cli(A,a)=V^*\pi(a)E(A)V=V^*\pi(a)V\mu_{\cli}(A)=\phi_{\cli}(a)\mu_{\cli}(A).$$\end{proof}
	\begin{rmrk}
		We may sometimes use \ref{product representation} of Theorem \ref{characterization of spectral} as the definition of a spectral instrument for convenience. Hereafter, we use the notation $\pi E$ to denote a spectral instrument on its respective domain and range:
        $$\pi E(A, a)=\pi(a)E(A), ~A\in \ox, a\in \cla .$$
\end{rmrk}
	\begin{dfn}[Irreducible instruments]
		A spectral instrument $\pi E:\ox\to CP(\cla,\bh)$ is said to be {\em irreducible}  if $\pi E$ has no proper invariant subspace. This is equivalent to requiring  $\{\pi E\}':= \{\pi (a)E(A): a\in \mathcal{A}, A\in \ox \}' =\mathbb{C}I$.
		
	\end{dfn}
	\begin{rmrk}\label{commutnt of spectral instruments are intersection of commutants of its marginals}
		For any spectral instrument $\pi E:\ox\to CP(\cla,\bh),$ one can check that $$\{\pi(\cla)E(\ox)\}'=\{\pi(\cla)\}'\cap \{E(\ox)\}'.$$
	\end{rmrk}
	\begin{ppsn}
		Let $\pi E:\ox \to CP(\cla,\bh)$ be a spectral instrument. Then TFAE,
		\begin{enumerate}[label=(\roman*)]
			\item $\pi E$ is irreducible;
			\item $\pi:\cla \to \bh$ is an irreducible representation of the $C^*$-algebra $\cla.$
		\end{enumerate}
		Furthermore,  in such a case, the associated spectral measure $E:\ox \to \bh$ is trivial, i.e., $E(\ox) \subseteq \{0, I_{\clh}\}$.
		
	\end{ppsn}
	\begin{proof}
		It follows from Remark \ref{commutnt of spectral instruments are intersection of commutants of its marginals}, that (ii) implies (i).
		To prove (i) $\implies$ (ii), we need to show that ${\pi(\cla)}' = \bbc I_\clh$. We will show that $E(\ox) \subseteq \{0, I_{\clh}\}$. This combined with Remark \ref{commutnt of spectral instruments are intersection of commutants of its marginals}, yields the desired result. Assume that there exists a proper projection $P\in E(\ox).$ Since $E$ is a spectral measure, every element in $E(\ox)$ is a projection, which are mutually orthogonal and they commute with unital $*$-homomorphism $\pi,$ i.e., $E(\ox)\subseteq E(\ox)'\cap\pi(\cla)' =\{\pi(\cla)E(\ox)\}'.$ However, by hypothesis, we have  $\{\pi(\cla)E(\ox)\}'=\bbc I_\clh ,$ which has only trivial projections.  As a consequence  $E(\ox) \subseteq \{0, I_{\clh}\}.$  
	\end{proof}
	\begin{dfn}
		A CP instrument $\cli$ is {\em concentrated}  on a measurable subset $S$ if $\cli(A)=\cli(A\cap S)$ for all $A\in\ox$.
	\end{dfn}

	\begin{ppsn}\label{zero sets of cli and piE are same}
		Let $\cli:\ox\to CP(\cla,\bh)$ be a CP instrument with the  minimal bi-dilation tuple $(\clk,\pi,E,V).$ Then for any $A\in\ox,~ \cli(A)=0$ if and only if $(\pi E)(A)=0.$ In particular, $\cli $ is concentrated on $S\in \ox$ if and only if $\pi E $ is concentrated on $S.$
	\end{ppsn}

	\begin{proof}
		The second assertion follows from the first one. So, it is enough to prove the first one.
		Let $\cli(A)=0.$ Then for any $a\in\cla,~B\in\ox$ and $h\in\clh,$ we get,
		\begin{equation*}
			\begin{split}
				\langle \pi(a)E(B)Vh, (\pi E)(A)(1_{\cla}) \pi(a)E(B)Vh)\rangle &=\langle h,V^*\pi(a^*)E(A\cap B)\pi(a)E(B)Vh)\rangle\\
				&=\langle h,\cli(A\cap B)(a^*a)h \rangle\\&\leq\langle h,\cli(A)(a^*a) h \rangle=0.
			\end{split}
		\end{equation*}
		Since $\{\pi(a)E(B)Vh:a\in\cla,B\in \ox ,,h\in\clh\}$ is a total set in $\clk,$ by the minimality condition, we conclude that $(\pi E)(A)(1_\cla)=0.$ Hence, we have  $\pi E(A)(a)=0, ~\forall a\in \cla .$ The converse is obvious. 
	\end{proof}
	\subsection{Radon-Nikodym type theorem}
	In classical measure theory, the existence and uniqueness of a Radon-Nikodym derivative of a ($\sigma$-finite) positive measure absolutely continuous with respect to another ($\sigma$-finite) positive measure, is a well-established result. In the case of POVMS (Theorem 2.8, \cite{manishcmp}) and CP maps (Theorem 1.4.2, \cite{arvesonsubalgebra}) the notion of Radon-Nikodym derivative has been introduced using the partial order defined by domination.
	An analogous  Radon-Nikodym type of theorem exists for quantum instruments as well. For two quantum instruments $\cli, \clj$, we  say $\clj $ is {\em dominated\/} by $\cli $ (denoted by $\clj\leq \cli$)
	if $\cli-\clj$ is a CP instrument. 
	\begin{thm}[Radon-Nikodym type theorem]\label{radonnikodymthm}
		Let $\cli:\ox\to CP(\cla,\bh)$ be an instrument with the minimal bi-dilation quadruple $(\clk,\pi,E,V)$. Then for any instrument $\clj : \ox \to CP(\cla,\bh)$, $\clj \leq \cli$ if and only if there exists a positive contraction 
		$D \in \{\pi(a)E(A):a\in\cla,A\in\ox\}'$ such that, \begin{equation}\label{RN}\clj(A,a)=V^*D \pi(a)E(A)V, ~\forall A\in\ox,~a\in\cla.\end{equation}
	\end{thm}
We will call the operator $D$ of this theorem as the {\em  Radon-Nikodym derivative} of $\clj $ with respect to $\cli.$ 
	
Motivated by the notion of pure completely positive (CP) maps for $\cst$-algebras, as introduced in \cite{arvesonsubalgebra}, we have the following definition for quantum instruments.
	
	\begin{dfn}[Pure instruments]\label{Pure instrument}
		Let $\cli:\ox\to CP(\cla,\bh)$ be a CP instrument. $\cli$ is  called a {\em pure instrument} if  any instrument $\clj:\ox\to CP(\cla,\bh)$ dominated by $\cli$ is of the form $\clj=t\cli$ for some $t\in [0,1].$
\end{dfn}

Here is a characterization of pure quantum instruments generalizing a similar result by  Arveson for pure CP maps.
	\begin{ppsn}\label{prop: characterization of pure instruments}
		Let $\cli:\ox\to CP(\cla,\bh)$ be a CP instrument with minimal bi-dilation quadruple $(\clk,\pi, E, V).$ Then $\cli$ is pure if and only if the spectral instrument $\pi E:\ox\to CP(\cla,\bh)$ is irreducible.
	\end{ppsn}
	\begin{proof}
		If the associated spectral instrument $\pi E:\ox\to Cp(\cla, \bk)$ in the minimal bi-dilation of $\cli$ is irreducible, then $\{\pi(\cla)E(\ox)\}'=\bbc I_{\clk}.$ As a consequence of Theorem \ref{radonnikodymthm}, for any instrument $\clj : \ox \to CP(\cla,\bh)$, $\clj \leq \cli$ we have that $\clj=t\cli$ for some $t\in [0,1],$ i.e. the instrument $\cli$ is pure. 
		Conversely,
		suppose the set $\{\pi(\cla)E(\ox)\}'$ contains a proper projection $P.$ Consider the instrument $\cli':\ox\to CP(\cla,\bh),$ defined by $$\cli'(A)(a)=V^*\pi(a)E(A)PV,~\forall ~a\in\cla,~A\in\ox.$$ It follows that $\cli'$ is  dominated by the instrument $\cli.$ Since the instrument is pure, we must have $\cli'=t\cli$ for some $t\in [0,1].$ Therefore, $$V^*\pi(a)E(A)PV=tV^*\pi(a)E(A)V$$ for all $a\in\cla,~A\in\ox,$ which implies that $P\in\{0,I_\clk\}.$ Hence, the spectral measure $\pi E $ is irreducible.	\end{proof}
	As a direct consequence of Proposition \ref{prop: characterization of pure instruments}, we obtain the following result.
	\begin{crlre}\label{cor: POVM marginal of pure instrument is trivial} The 
		POVM marginal of any pure instrument is trivial.
	\end{crlre}
	\begin{dfn}[Compression of instruments]
		Let $\cli_i:\ox\to CP(\cla,\clb(\clh)_i),i=1,2$ be two CP instruments. Then $\cli_2$ is said to be a compression of  $\cli_1$,  if there exists an isometry $W:\clh_2\to\clh_1$ such that $\cli_2(A)=W^*\cli_1(A)W,~\text{for all}~A\in\ox.$
		
	\end{dfn}
	\begin{rmrk}\label{rmk: POVM marginal of compression pure instrument is trivial}
		The compressions of a pure instrument are pure. Furthermore, by Corollary \ref{cor: POVM marginal of pure instrument is trivial},  the POVM marginals of compressions of pure instruments are  trivial. \end{rmrk}
        We record the following elementary observation for subsequent use.
        \begin{rmrk} Let $\cli _1, \cli _2$ be two quantum instruments, which are compressions of a common  spectral instrument $\pi E$. 		Suppose $(\clk,\pi,E,V_i)$ denote the minimal bi-dilation tuple of $\cli_i,i=1,2$. 
        Then one can verify that $\cli_2$ is a compression of $\cli_1$ iff $V_2V_2^*\leq V_1V_1^*,$ which is equivalent to, $\text{range}(V_2)\subseteq\text{range}(V_1).$ 
	\end{rmrk}
	\subsection{Sub-minimal dilations}

Within the framework of the bivariate realization outlined in \ref{bivariate-realization}, it becomes evident that, unlike completely positive (CP) maps or POVMs, quantum instruments admit multiple types of dilations. However, a closer investigation reveals that all such dilations ultimately reduce to the fundamental bi-dilation structure. Our next objective is to explore these variants, collectively referred to as \emph{sub-minimal dilations}, a notion originally introduced in \cite{pellonpaapieces}.

While the authors in \cite{pellonpaapieces} introduced the notion of sub-minimal dilations in the context of instruments arising from completely positive maps defined on the von Neumann tensor product of algebras into $\mathcal{B}(\mathcal{H})$, their focus was primarily on understanding the connection between the marginals and the joint CP map. The intrinsic relationships between two sub-minimal dilations and the minimal bi-dilation was not their focus. In this work, we reinterpret and generalize the notion of sub-minimal dilations in a more abstract setting, offering what we believe is a more systematic treatment of the topic.

	Based on their construction methods, these dilations are classified into two types: CP sub-minimal dilations and POVM sub-minimal dilations. Briefly speaking, in the CP sub-minimal dilation, the CP marginal is dilated using the Stinespring representation, ignoring the POVM marginal. Similarly, in the POVM sub-minimal dilation the POVM marginal is dilated with help from the Naimark dilation, disregarding the other marginal. For clarity and convenience, we provide a concise overview of these constructions below.

	\begin{thm}[CP sub-minimal dilation]\label{subminimal-1}
		Let $\cli:\ox\to CP(\cla,\bh)$ be a CP instrument. Then there exists a Hilbert space $\clk_1$, a unital $*$-homomorphism $\pi_1: \cla \to \mathcal{B}$, a normalized POVM $\mu: \ox \to \pi_1(\cla)' \subset \clb(\clk_1)$ and a bounded linear map $V_1: \clh \to \clk_1$ such that, \begin{equation}
\cli(A,a)=V^*_1\pi_1(a)\mu(A)V_1\quad\quad\text{and}\quad\quad \mu(A)\pi_1(a)=\pi_1(a)\mu(A),\end{equation}
		for all $a \in \cla$, $A \in \clo(X)$. Moreover, it satisfies the minimality condition: $[\pi_1(\cla)V_1\clh]=\clk_1$. 
	\end{thm}
\begin{proof}
		Let $(\clk_1,\pi_1,V_1)$ be the minimal Stinespring dilation of the associated CP map $\phi_\cli$ of the instrument $\cli, $
        so that $\cli (\cdot )=V_1^*\pi _1(\cdot )V_1, $ and $\clk_1=[\pi _1(\cla )V_1(\clh )].$
		For every $a\in\cla$ and $A\in\ox$ it holds that $\cli(A,a)\leq\cli(X,a).$ By the Radon Nikodym theorem for CP maps, there exists $D\in\pi_1(\cla)'$ with $0\leq D\leq I_{\clk_1}$ such that \begin{equation*}
			\cli(A,a)=V_1^*\pi_1(a)DV_1,~ for~all~a\in\cla.
		\end{equation*} 
		This implies that every element of $\clo(X)$ corresponds to a positive contraction of $\pi_1(\cla)'.$ Let, $\mu:\clo(X)\to\pi_1(\cla)'$ denote this correspondence. It follows directly from the definition of $\mu$ that it commutes with the $*$-homomorphism $\pi_1.$ To establish that $\mu$ is a POVM, it remains to verify the following:
		\begin{equation*}
			\mu(\cup_i A_i)=\sum_i \mu(A_i),
		\end{equation*}
		for any countable collection of disjoint subsets $\{A_i\}\subset \clo(X)$, where the sum on the right-hand side converges in WOT. Since the collection of vectors of the form $\pi_1(a)V_1h$ forms a total set for $\clk_1$, it suffices to prove that,
		\begin{eqnarray}\label{povmproof1}
			\langle \pi_1(a)V_1h,\mu(\cup A_i)\pi_1(b)V_1k\rangle =\sum_i\langle\pi_1(a)V_1h,\mu(A_i)\pi_1(b)V_1k\rangle, 
		\end{eqnarray}
		for all $a,b\in \cla$ and $h,k\in \clh $. Equation \ref{povmproof1} is an immediate consequence of the fact that $\cli$ is an instrument. Consequently, we have that
		$\cli(A,a)=V_1^*\pi_1(a)\mu(A)V_1~,\forall ~a\in\cla,A\in\clo(X).$ 
	\end{proof}
	\begin{rmrk}\label{sub-minimal dilation dilates to bi-dilation}
		Furthermore, if we consider the minimal Naimark dilation of the normalized POVM $\mu$ that appears in the subminimal dilation described above, we obtain, $$\cli=V_1^*\pi_1W_1^*E_1W_1V_1,$$ where $(\widetilde{\clk_1},E_1, W_1)$ is the minimal Naimark tuple for $\mu$, satisfying $\mu (\cdot )= W_1^*E_1(\cdot )W_1$ with  $$\widetilde{\clk_1}=[E_1(\ox )W_1\clk _1]=[E_1(\ox)W_1\pi_1(\cla )V_1(\clh )].$$

		Following an argument analogous to the construction of the bi-dilation, one sees that the tuple $(\widetilde{\clk_1},\widetilde{\pi_1}, E_1, W_1V_1)$ forms a minimal bi-dilation tuple of the instrument $\cli,$ where $\widetilde{\pi_1}:\cla\to\clb(\widetilde{\clk_1})$ is the unital $*$-homomorphism commuting with $E_1$ given by, \begin{eqnarray*}
		  \widetilde{\pi_1}(b)(E_1(A)W_1\pi_1(a)V_1h)=E_1(A)W_1\pi_1(ba)V_1h, ~\forall A\in\ox,~a,b\in\cla,~h\in\clh.
		\end{eqnarray*} This calculation clarifies that the notion of subminimal dilation does not provide much structural insight beyond the framework of minimal bi-dilation.
	\end{rmrk}

	We call the quadruple $(\clk_1,\pi_1,\mu,V_1)$ in Theorem \ref{subminimal-1}, the CP subminimal dilation quadruple of $\cli$ and adopt this notation throughout the paper. Next we introduce another type of subminimal dilation. Instead of Stinespring dilation of the associated CP map, we start with the Naimark dilation of the associated POVM $\mu_{\cli}.$  The proof strategy mirrors that of Theorem \ref{subminimal-1} and hence we omit the proof. However we present the formal statement.
	\begin{thm}[POVM Subminimal dilation]\label{subminimal-2}
		Let $\cli :\ox\to CP(\cla,\bh)$ be a CP instrument. Then there exists a Hilbert space $\clk_2,$ a spectral measure $E_2: \clo(X) \to \clb(\clk_2),$ a UCP map $\phi: \cla \to E_2(\clo(X))' \subset\clb(\clk_2) $, and a linear map $V_2: \clh \to \clk_2$ such that,
		$$\cli(A,a)=V^*_2\phi(a)E_2(A)V_2\quad\quad\text{and}\quad\quad E_2(A)\phi(a)=\phi(a)E_2(A),$$
		for all $a \in \cla$, $A \in \clo(X)$. Moreover, it satisfies the minimality condition: $[E_2(\ox)V_2\clh]=\clk_2$. 
	\end{thm}
	\begin{rmrk}\label{sub-minimal dilation dilates to bi-dilation-2}
   Similar to Remark \ref{sub-minimal dilation dilates to bi-dilation}, in this case as well, if we further take the minimal Stinespring dilation $(\widetilde{\clk_2},\pi_2,W_2)$ of the UCP map $\phi,$ we obtain a minimal bi-dilation of the instrument $\cli,$ given by, $(\widetilde{\clk_2},\pi_2,\widetilde{E}_2,W_2V_2),$ where, $\widetilde{\clk_2}=\overline{\mbox{span}}\{\pi_2(a)W_2(E_2(A)V_2(h)):a\in\cla,~A\in\ox,~h\in\clh\}$ and $\widetilde{E}_2:\ox\to\clb(\widetilde{\clk}_2)$ is a spectral measure which commutes with the $*$-homomorphism ${\pi}_2$ given by,$$\widetilde{E}_2(B)(\pi_2(a)W_2(E_2(A)V_2(h)))=\pi_2(a)W_2(E_2(A\cap B)V_2(h)),~a\in\cla,~A,B\in\ox,~h\in\clh.$$
        \end{rmrk}
	We call the quadruple $(\clk_2,\phi,E_2,V_2)$ in Theorem \ref{subminimal-2}, as the POVM subminimal dilation quadruple of $\cli$ and we fix this notation throughout the paper.
    It is not hard to write down examples where all the three dilations are distinct.

	\subsection{Recovering subminimal dilations from the minimal bi-dilation}\label{Relation between Sub-minimal dilation and minimal Bi-dilation}
	From the previous section on sub-minimal dilation (Remarks \ref{sub-minimal dilation dilates to bi-dilation} and \ref{sub-minimal dilation dilates to bi-dilation-2}), it is clear that every sub-minimal dilation naturally leads to the minimal bi-dilation of an instrument. Conversely, suppose we begin with the minimal dilation quadruple $(\clk,\pi, E, V) $ of an instrument $\cli.$ 
	Consider the orthogonal projections $P_1,P_2\in\bk$ with ranges as the sub-spaces $\clk_1,\clk_2$ of the minimal dilation space $\clk,$ defined as $$\clk_1=[\pi(\cla)V(\clh)] ~\text{and}~ \clk_2=[E(\ox)V(\clh)].$$ Since $\clk_1$ is invariant under the representation $\pi,$ the compression $P_1\pi P_1$ induces a sub-representation of $\pi.$ Then it is immediate from the definition of $\clk_1$ that the tuple $(\clk_1,P_1\pi P_1,V)$ forms a minimal Stinespring triple for the CP marginal $$\phi_{\cli}(a)= V^*\pi(a)V,~\text{for all}~a\in\cla.$$ Moreover, the POVM $\mu$ appears in the CP sub-minimal dilation Theorem \ref{subminimal-1} can be identified as $$\mu(A)=P_1E(A)P_1,~\forall~A\in\ox.$$ Thus, we conclude that starting from the minimal bi-dilation quadruple $(\clk, \pi, E, V)$ of an instrument $\cli$, we can explicitly identify the CP sub-minimal dilation as $(\clk_1, P_1 \pi P_1, P_1 E P_1, V).$ 
	By a similar line of reasoning, we can recover the POVM subminimal dilation as $(\clk_2, P_2 \pi P_2, P_2 E P_2, V)$ from the minimal bi-dilation quadruple $(\clk,\pi,E,V)$ of the instrument $\cli$.

	\subsection{Decomposable instruments}
    In classical measure theory, a product measure combines two measure spaces into a single joint space in a consistent and well-understood manner. Motivated by the discussion in Section \ref{Introduction}, we now consider the quantum analogue of this concept. In \cite{ozawa}, Ozawa introduced the notion of \emph{decomposable} instruments, extending the idea of product measures to the non-commutative setting of operator algebras. Such instruments are those whose statistical structure can be factored along two components, paralleling the factorization of product measures into their marginals in the classical case.  This interpretation is supported by the structural insights provided in Remark \ref{bivariate-realization}. We now formally present the definition of decomposable instruments:

	\begin{dfn}[Decomposable instruments]\label{Decomposableinstruments}
		Let $\cli:\ox\to CP(\cla,\bh)$ be a CP instrument. Then $\cli$ is said to be a decomposable instrument if
		\begin{equation*}
			\cli(A,a)=\phi_\cli(a)\,\mu_\cli(A),\quad \forall a \in \cla, ~A \in \ox,
		\end{equation*}
		where $\phi_\cli$ and $\mu_\cli$ are respectively the CP map and the POVM associated with the instrument $\cli .$ 
	\end{dfn}
	It is evident that if $\cli:\ox\times\cla\to\bh$ is a decomposable instrument, then $\cli(A,a)=\phi_\cli(a)\mu_\cli(A)=\mu_\cli(A)\phi_\cli(a)$, $\forall a \in \cla, ~A \in \clo(X).$

	We will now present a characterization of decomposable instruments utilizing the concept of sub-minimal dilations of quantum instruments.
	\begin{thm}\label{decomposibleequivalent-1}
		Let $\cli:\clo(X)\times\cla\to\bh$ be a CP instrument with the minimal bi-dilation quadruple $(\clk,\pi,E,V).$ Let $P_1,P_2$ be the orthogonal projections onto the sub-spaces $[\pi(\cla)V(\clh)]$ and $[E(\ox)V(\clh)]$ of the dilation space $\clk.$ Then TFAE: 
		\begin{enumerate}[label=(\roman*)]
			\item $\cli$ is decomposable,
			\item $P_1E(A)P_1VV^*=VV^*P_1E(A)P_1VV^*=VV^*P_1E(A)P_1,$ for all $A\in\ox,$
			\item $P_2\pi(a)P_2VV^*=VV^*P_2\pi(a)P_2VV^*=VV^*\pi(a)P_2,$ for all~$a\in\cla.$
		\end{enumerate}

	\end{thm}
	\begin{proof}
		From the definition \ref{Decomposableinstruments} we have, $\cli$ is decomposable if and only if 
		\begin{eqnarray}\label{decomposibleequation}
			\cli(A,a)= \phi_\cli(a)\mu_\cli(A),\quad \forall a \in \cla, ~A \in \clo(X).
		\end{eqnarray}
		Using the CP sub-minimal dilation of $\cli,$ equation \ref{decomposibleequation} is equivalent to:
		\begin{eqnarray*}
		& &	V^*P_1\pi(a) P_1E(A)P_1V=V^*P_1\pi P_1(a)VV^*P_1E(A)P_1V \\
            &\iff& V^*P_1E(A)P_1(1-VV^*)P_1\pi(a) P_1V=0\\
			&\iff& P_1 E(A)P_1 V=VV^*P_1E(A)P_1V\\&\iff &P_1E(A)P_1VV^*=VV ^*P_1E(A)P_1VV^*,
		\end{eqnarray*}
		for all $A\in\ox,~a\in\cla.$ This establishes the equivalence (i)$\iff$ (ii).
		Following a similar line of reasoning one can establish the equivalence (i)$\iff$ (iii). Combining these two results, we conclude the equivalence between (ii) and (iii).
	\end{proof}

	As a consequence of Theorem \ref{decomposibleequivalent-1}, we obtain the following corollary, originally established by Ozawa (Proposition 4.3 in \cite{ozawa}). Making use of this it is possible to give an alternative proof of Theorem \ref{characterization of spectral}.

	\begin{crlre}\label{decomposable-sufficientcondition}
		A CP instrument $\cli:\clo(X)\times\cla\to\bh$ is decomposable if either $\phi_\cli$ is a $*$-homomorphism or $\mu_\cli$ is a spectral measure.
		
	\end{crlre}

	\subsection{Atomic and non-atomic instruments}
	
	In classical measure theory atomic and non-atomic measures have been widely studied.  
	To better understand the structure of POVMs Ramsey, Plosker, et al. \cite{plosker_ramsey_povmintegratiion_2} extended these notions to the quantum setting of POVMs. More recently, the authors of \cite{manishcmp} utilized the decomposition of POVMs into atomic
and non-atomic POVMs to analyze the $C^*$-extreme points. Here we extend these ideas to the framework of quantum instruments, aiming to recast and build upon some of these results. 
	
	\begin{dfn}\label{definition of atomic and non-atomic instruments}
		Let $\cli:\ox\to CP_{\clh}(\clh)$ be a CP instrument. A subset $ A\in\ox$ is called an
		\emph{atom} for $\cli$ if $\cli( A)\neq0$ and whenever $ B\subseteq A$
		in $\ox$,
			either  $\cli( B)=0$ or $\cli( B)=\cli( A).$
		A CP instrument $\cli$ is called \emph{atomic} if every $ A\in\ox$ with $\cli( A)\neq0$ contains an atom. A CP instrument $\cli$ is called \emph{non-atomic} if it has no atom.
	\end{dfn}

	\begin{rmrk}
		From the definition, it is straight forward to verify that a set $A$ is an atom for an instrument $\cli$ iff it is an atom for the associated POVM $\mu_{\cli}.$  
	\end{rmrk}

	It is a well-known fact that every finite (and more generally, every 
	$\sigma$-finite) positive measure decomposes uniquely into the sum of an atomic and a non-atomic positive measure (see \cite{johnson}). In an analogous manner, every POVM admits a unique decomposition as the sum of an atomic POVM and a non-atomic POVM, as established in \cite{plosker_ramsey_povmintegratiion_2}. While the proof in \cite{plosker_ramsey_povmintegratiion_2} was formulated for POVMs on locally compact Hausdorff spaces, it was later fully generalized to arbitrary measurable spaces in \cite{manishcmp}. We extend this result to the setting of instruments and state the corresponding decomposition theorem in this context. The proof closely follows the classical case as presented in \cite{johnson} and we skip it.

	\begin{thm}\label{thm:every instruments decomposes as a sum of atomic and non atomic instruments}
		Every Quantum instrument decomposes uniquely as a sum of an atomic CP instrument and a non-atomic CP instrument.
	\end{thm}

	We conclude this discussion with an useful observation on atoms of instruments.
	\begin{ppsn}\label{cli is atomic iff piE is atomic}
		Let $\cli:\ox\to CP(\cla,\bh)$ be a CP instrument with the minimal bi-dilation tuple
		$(\clk,\pi,E,V)$. Then a subset $ A\in\ox$ is an atom for $\cli$ if
		and only if $ A$ is an atom for the spectral instrument $\pi E$. In particular, $\cli$ is
		atomic (non-atomic) if and only if $\pi E$ is atomic (non-atomic).
	\end{ppsn}
	\begin{proof}
		For any subset $ A\in\ox$,  $ A$ is an atom for $\cli$ if and  only
		if $\cli(A)\neq 0$ and for each $ A'\subseteq A$ in $\ox$, we have either $\cli( A')=0$ or $\cli( A\setminus
		A')=0$. Equivalently,  $(\pi E)(A)\neq 0$ and we have  either $(\pi E)( A')=0$ or $(\pi E)( A\setminus
		A')=0$. From Proposition \ref{zero sets of cli and piE are same} this  is same as saying that $ A$ is an atom for $\pi E$. The second assertion  follows easily from the first.
	\end{proof}

	\section{Extremality and \texorpdfstring{$\cst$}{C*}-Convexity of Instruments}\label{Extremality and $\cst$-Convexity of Instruments}
	\subsection{Extreme instruments}
	
	The convexity structure of unital completely positive (UCP) maps and normalized POVMs has attracted considerable attention in the literature. It was W. Arveson who first established an abstract characterization of extreme UCP maps, presented as Theorem 1.4.6 in \cite{arvesonsubalgebra}. Later, M. D. Choi revisited this result in the setting of matrix algebras \cite{Choi}. For characterizations of extreme points of normalized POVMs, we refer to \cite{pellonapaa1}, \cite{Douglas_Plosker_and_Smith}, and \cite{Heinosaari_Teiko_and_Pellonpa}.
    
	The notion of quantum instruments includes UCP maps and POVMs as special cases and it is worth looking at their convexity structure and extreme points. The set $I_{\clh}(X,\cla)$, consisting of all normalized CP instruments on 
	X with values in $CP(\cla,\bh)$, clearly forms a convex set. Motivated by Arveson’s abstract characterization of extreme UCP maps in \cite{arvesonsubalgebra}, one finds that extreme CP instruments admit a similar form of abstract characterization. This has already been observed by several authors: in particular, Pellonpää \cite{pellonapaa1} studied this using the framework of direct integrals, and D’Ariano et al. \cite{Ariano} discussed it in the finite-dimensional setting. We state the result below in our general setting without proof, as its derivation follows similar lines of reasoning to those used in Theorem 1.4.6 of \cite{arvesonsubalgebra}.

	\begin{thm}[Extreme point condition]\label{thm:extrme point criterion for instruments}
		Suppose that $\cli\in I_{\clh}(X,\cla)$  has  the minimal dilation tuple $(\clk,\pi,E,V)$.  Then a necessary and sufficient criterion
		for  $\cli$ to be extreme in $I_{\clh}(X,\cla)$ is that the map $D\mapsto V^*DV$
		from  $\{\pi(\cla)E(\ox)\}'$ to $\bh $ is injective.
	\end{thm}

	As a direct consequence of Theorem \ref{thm:extrme point criterion for instruments}, we obtain the following corollary.

	\begin{crlre}
		Every spectral instrument is an extreme point in $I_{\clh}(X,\cla).$
	\end{crlre}
	
    Next, we present a characterization of extremity, in terms of the Choi–Kraus decomposition (Proposition \ref{Choi-Kraus}), for CP instruments defined on finite sets. This result follows as a consequence of Theorem \ref{thm:extrme point criterion for instruments} and extends Choi’s characterization of extreme UCP maps on matrix algebras ( Theorem 5 of \cite{Choi}).

	\begin{crlre}\label{Choi-Kraus characterization of extreme instruments}
		Let $\cli:\ox\to CP(M_d, M_k)$ be a CP instrument, where $X=\{1,\cdots,n\}$ be a finite set. Suppose, $$\cli(i,a)=\sum_j{V_j^i}^*a{V_j^i},~\forall~a \in M_d,$$ is a minimal Choi–Kraus decomposition of $\cli,$ with $\{{V_j^i}\}$ a set of matrices in $M_{k \times d}$ for each $i.$  Then $\cli$ is extreme if and only if the set $\{{V_j^i}^*{V_j^i}\}$ is linearly independent for each $i\in\{1,\cdots,n\}$.
	\end{crlre}

	Corollary \ref{Choi-Kraus characterization of extreme instruments} enables the construction of numerous extreme instruments that are not spectral in nature. As an illustration, consider the following example of a unital completely positive (UCP) instrument that is extreme, yet not spectral.

	\begin{xmpl}\label{exmp: extreme instrument with non-extreme marginals}
		Fix $0<t<1.$ Let $\mu:\ox\to M_2$ be a POVM on $X=\{1,2\},$ defined by $$\mu(1)=t E_{11} +(1-t)E_{22},~~\mu(2)=(1-t)E_{11} + tE_{22}.$$ where $E_{ii}$ denote the matrix unit with $1$ in the $(i,i)$-th entry and $0$ elsewhere. Consider the instrument $\cli:\ox\to CP(M_2, M_2) $ given by,
		$$\cli(i)(A)=\mu(i)^\frac{1}{2}A\mu(i)^\frac{1}{2},~\text{for all}~A\in M_2.$$ \end{xmpl}
It follows immediately from Corollary \ref{Choi-Kraus characterization of extreme instruments} that $\cli$ is an extreme unital completely positive (UCP) instrument.
	Moreover, observe that the POVM marginal  $\mu_\cli$ of 
	$\cli$ coincides with the original POVM $\mu$. Since 
	$\mu$ is not spectral, it follows that $\cli$ cannot be a spectral instrument either.

	 Here is an alternative abstract characterization of the extremity of instruments, in terms of the natural partial order on instruments. Recall that an instrument $\clj $ is {\em dominated} by an instrument $\cli $, if $\cli (A)-\clj (A)$  is CP for every $A$ in $\ox .$
	
	\begin{thm}\label{thm: extreme point condition by Bhat}
		Let $\cli:\ox\to CP(\cla,\bh)$ be a UCP instrument. Then $\cli$ is extreme if and only if   $\clj_1=\clj_2$, for any two CP instruments $\clj_1,\clj_2$ dominated by $\cli$ with  $\clj_1(X)=\clj_2(X)$.  	\end{thm}
	\begin{proof}
		To prove the `if' part, we consider the UCP instruments, $\cli_1=\cli-\clj_1+\clj_2$ and $\cli_2=\cli-\clj_2+\clj_1$. It is clear that if $\clj_1\neq\clj_2,$ then we have a proper convex combination for the extreme instrument $\cli.$ 
		Next we assume the converse part in the hypothesis. If $\cli$ is not extreme then there exists a proper convex combination, $\cli=\sum p_i\cli_i,$ for $\cli.$ Consider the instruments, $\clj_1=\cli-(p_1\wedge p_2)\cli_1$ and $\clj_2=\cli-(p_1\wedge p_2)\cli_2,$ where $p_1\wedge p_2=\text{min}\{p_1, p_2\}.$ Then it is straight forward to verify that $\clj_1(X)=\clj_2(X)$ and both of them are dominated by $\cli.$ By the hypothesis, $\clj_1=\clj_2,$ which further implies that $\cli_1=\cli_2.$ Therefore, $\cli$ must be extreme.
	\end{proof}

	\subsection{\texorpdfstring{$\cst$}{C*}-extreme instruments}
	Inspired by the literature on $\cst$-convexity of POVMs and unital completely positive (UCP) maps, we introduce the concepts of $\cst$-convexity and $\cst$-extreme points in the framework of instruments. In this context, we also explore abstract characterizations of $\cst$-extremity for instruments. To begin with, we formally define C*-convexity for the collection of all normalized instruments $I_\mathcal{H}(X,\cla).$
	
	\begin{dfn} [$C^*$-convexity]\label{cstarconvexity}
		For any $\cli_i\in I_\clh(X,\cla)$ and $T_i\in\bh$, $1\leq i\leq  n$ with $\sum_{i=1}^n T_i^*T_i=I_\mathcal{H}$, a sum of the form
		\begin{equation}\label{C^*sum}
			\cli(\cdot)=   \sum_{i=1}^nT_i^*\cli_i(\cdot)T_i
		\end{equation}
		is called a \emph{$\cst$-convex combination}\index{$C^*$-convex combination} for $\cli$. The operators $T_i$'s here
		are called {\em $\cst$-coefficients}\index{$C^*$-coefficients}. When $T_i$'s are invertible,
		the sum in \eqref{C^*sum} is called a
		\emph{proper $\cst$-convex combination}\index{$C^*$-convex combination!proper-} for $\cli$.
	\end{dfn}
	
	Observe that $I_\clh(X,\cla)$ is a {\em $\cst$-convex set}\index{$C^*$-convex set} in the sense that it is closed
	under $\cst$-convex   combinations i.e. $\sum_{i=1}^n{T_i}^*\cli_i(\cdot)T_i\in I_\clh(X,\cla)$, whenever $\cli_i\in I_\clh(X,\cla)$ and $T_i\in\clb(\clh)$ satisfying $\sum_{i=1}^n{T_i}^*T_i=I_\clh$. 
	
	\begin{dfn}[$C^*$-extreme point]\label{cstarextreme}
		An instrument $\cli:\ox\to CP(\cla,\bh)$ is called a \emph{$C^*$-extreme point}\index{$C^*$-extreme point} in
		$I_\clh(X,\cla)$ if, whenever  $\sum_{i=1}^nT_i^*\cli_i(\cdot)T_i$
		is a proper $C^*$-convex combination of $\cli$, then each $\cli_i$ is unitarily equivalent to $\cli$ i.e. there are unitary operators $U_i\in\clb(\clh)$
		such that $\cli_i(\cdot)=U_i^*\cli(\cdot)U_i $ for $1\leq i\leq n.$
	\end{dfn}

	\subsection{Abstract characterizations of \texorpdfstring{$\cst$}{C*}-extreme points}
	Farenick and Zhou \cite{farenickpams}, taking cue from Arveson's extreme point criteria for UCP maps, introduced an abstract characterization of $\cst$-extreme points for unital completely positive maps. Later, Bhat and Kumar \cite{manishjfa}, adapted this characterization to the setting of $\cst$-extreme points of normalized POVMs. In this work, we extend these ideas further by presenting an abstract characterization of C*-extreme instruments. Our proof is a direct adaptation of the method used for the POVM case   \cite{manishjfa} and we skip it.

	\begin{thm}\label{abstractcharacterizationinstrument}
		Let $\cli:\ox \to CP(\cla,\bh)$ be a CP instrument with the minimal bi-dilation quadruple $(\clk,\pi,E,V)$. Then $\cli$ is a $C^*$-extreme point in $I_\mathcal{H}(X,\cla)$ if and only if for any positive operator $D\in \{\pi(\cla)E(\ox)\}'$ with $V^*DV$ being invertible, there exists a co-isometry $U\in \{\pi(\cla)E(\ox)\}'$
		(i.e. $UU^*=I_{\mathcal{H}})$ satisfying $U^*U\sqd=\sqd$ and an invertible operator
		$S\in\bh$ such that $UD^{1/2}V=VS$.
	\end{thm}
 The following is an immediate corollary of Theorem \ref{abstractcharacterizationinstrument}.
	
	\begin{crlre}
		Every spectral instrument is a $\cst$-extreme point in $I_{\clh}(X,\cla).$
	\end{crlre}
	\begin{proof}
		If $\cli$ is a spectral instrument then the minimal bi-dilation for $\cli$ can be taken to be $(\pi, E, I_\clh, \clh )$. For positive $D\in \{\pi(\cla)E(\ox)\}'$ with $D (=I_\clh^*DI_\clh)$ invertible, we can take $U=I_\clh$ and $S=D^{1/2}$ to satisfy the
		criterion.
	\end{proof}

	We now present an alternative abstract characterization of 
	$\cst$-extreme points, adapted from a corollary earlier formulated for CP maps by Bhat and Kumar in \cite{manishjfa}. This result plays a crucial role in the analysis of $\cst$-extremity in the setting of instruments.
	
	\begin{crlre}\label{factorizationcorollary}
		Let $\cli:\ox \to CP(\cla,\bh) $ be an instrument with  the minimal bi-dilation $(\clk,\pi,E,V)$. Then $\cli$ is a $C^*$-extreme point in $I_\clh(X,\cla)$ if and only if for any positive operator $D\in \{\pi(\cla)E(\ox)\}'$ with $V^*DV$ being invertible, there exists $S\in \{\pi(\cla)E(\ox)\}'$ such that $D=S^*S$, $ SVV^*=VV^*SVV^*$  and $V^*SV$ is invertible (i.e. $S(V\clh)\subseteq V\clh$ and $S_{|_{V\clh}}$ is invertible).  
	\end{crlre}
	
	Analogous to Theorem \ref{thm: extreme point condition by Bhat}, which characterizes extreme instruments, we now present a result concerning $\cst$-extremity. This criterion, originally due to Farenick and Zhou, was established in the context of completely positive maps in \cite{farenickpams}. 
	\begin{crlre}\label{Zhou type characterization of $\cst$-extreme instruments}
		Let $\cli\in I_\clh(X,\cla)$. Then $\cli$ is $C^*$-extreme in $I_\clh(X,\cla)$ if and only if for any CP
		instrument $\clj:\ox\to CP(\cla,\bh)$ with $\clj\leq \cli$ and $\clj(X,1_\cla)$ invertible, then
		there exists an invertible operator $S\in\clb(\clh)$ such that
		$\clj(A,a)=S^*\cli(A,a)S$ for all $a \in \cla,~A\in\ox$.
	\end{crlre}

	\subsection{Direct sums of pure instruments}
	
	We now turn to the question of when the direct sum of two $\cst$-extreme instruments remains $\cst$-extreme. This discussion is motivated by the necessary and sufficient condition established in \cite{manishjfa} for the direct sum of unital completely positive maps to be $\cst$-extreme. We present a brief outline of the corresponding results for instruments. As a first step, we define the notion of the direct sum of instruments.
	\begin{dfn}[Direct sums of instruments]
		
		For any countable collection of CP instruments $\{\cli_i:\ox\to CP(\cla,\clb(\clh_i))\}_{i\in\Lambda},$ their direct sum $\oplus_{i\in\Lambda}\cli_i$ is the instrument $\oplus_{i\in\Lambda}\cli_i:\ox \to CP(\cla,\bh),$ defined by $(\oplus_{i\in\Lambda}\cli_i)(A)(a)=\oplus_{i\in\Lambda}\cli_i(A)(a)$, for all~$A\in\clo(X),a\in\cla,$ where $\clh=\oplus_{i\in\Lambda}\clh_i.$  
	\end{dfn}

	Motivated by the concept of mutually disjoint UCP maps, we introduce the notion of disjointness in the context of instruments. We begin by defining when two spectral instruments are said to be mutually disjoint, and subsequently extend this notion to general instruments.

	\begin{dfn}[Mutually disjoint spectral instruments]
		Two spectral instruments $\pi_iE_i:\ox \to CP_{\clh_{i}}(\cla),i=1,2$ are said to be mutually disjoint if there does not exist any non-zero $T\in \clb(\clh_1,\clh_2)$ satisfying $T\pi_1(a)E_1(A)=\pi_2(a)E_2(A)T$ for all $a\in\cla,A\in\clo(X).$  
	\end{dfn}
	Now we are in a position to extend this definition for instruments.

	\begin{dfn}
		Let $\cli_i:\ox\to CP_{\clh_{i}}(\cla),i=1,2$ be two CP instruments with respective minimal bi-dilationn $(\pi_i,E_i,V_i,\clk_i).$ Then $\cli_1$ and  $\cli_2$ are said to be disjoint if $\pi_1E_1$ is disjoint to $\pi_2E_2.$\end{dfn}
	
	The following proposition is easy to prove using the characterizations of $C^*$-extrmity proved above. 
	\begin{ppsn}\label{mutually disjoint directsum}
		Let $\{ \cli_i:\ox\to CP_{\clh_{i}}(\cla) \}_{i\in\Lambda}$ be a countable collection of mutually disjoint instruments. Then  $\cli=\oplus_{i\in\Lambda}\cli_i$ is $C^*$-extreme iff each $\cli_i$ is $C^*$-extreme.
	\end{ppsn}
	
	The following theorem is a direct adaptation of Theorem 3.7 in \cite{manishjfa} in the context of instruments.
	
	\begin{thm}\label{direct sum of pure instruments}
		Let $\cli$ be a direct sum of pure UCP instruments, so that $\cli$ is unitarily equivalent to $\bigoplus_{\alpha\in\Gamma}\bigoplus_{i\in\Lambda_\alpha}\clj_\alpha^i(\cdot)\otimes I_{\clk_\alpha^i}$, where $\clk_\alpha^i$ is a Hilbert space and $\clj_\alpha^i$ is a pure UCP instrument with minimal bi-dilation tuple $(\pi_\alpha, E_\alpha, V_\alpha^i,\clh_{\pi_\alpha E_\alpha})$ such that $\clj_\alpha^i$ is non-unitarily equivalent to $\clj_\alpha^j$ for each $i\neq j$ in $\Lambda _\alpha$, $\alpha\in\Gamma$, and $\pi_\alpha E_\alpha$ is disjoint to $\pi_\beta E_\beta$ for $\alpha\neq \beta$. Then $\cli$ is $\cst$-extreme  in $I_\clh(X,\cla)$ if and only if the following holds for each $\alpha\in\Gamma$ :		\begin{enumerate}
			\item\label{V_i are nests} $\{\text{range}~(V_\alpha^i)\}_{i\in\Lambda_\alpha}$ is a nest in $\clh_{\pi_\alpha E_\alpha}$, which makes $\Lambda_\alpha$ a totally ordered set, and
			\item\label{completion of V_i is countable} if $\cll_\alpha^i=\oplus_{j\leq i}\clk_\alpha^i$ for $i\in\Lambda_\alpha$, then the completion of the nest $\{\cll_\alpha^i\}_{i\in\Lambda_\alpha}$ in $\oplus_{i\in\Lambda_\alpha}\clk_\alpha^i$ is countable.
		\end{enumerate}
	\end{thm}

	\subsection{Characterization and structure of \texorpdfstring{$\cst$}{C*}-extreme instruments in finite dimensions}

	In the context of completely positive (CP) maps and POVMs on finite-dimensional Hilbert spaces, every $\cst$-extreme point also qualifies as an extreme point in the classical convex sense. Following the approach presented in Proposition 2.1 of \cite{Douglas_Plosker_and_Smith}, this correspondence can be established for instruments. However, it remains unclear whether the same holds when $\clh$ is infinite-dimensional.

\begin{thm}\label{cstarextremeimpliesextreme}
		If $\text{dim}~(\clh)<\infty$, then every $C^*$-extreme point in $I_\clh(X,\cla)$  is an extreme point.
	\end{thm}

    \begin{proof}
        Suppose $\mathcal I=t\,\mathcal I_1+(1-t)\,\mathcal I_2$ with $0<t<1$ and $\mathcal I_i$ UCP instruments.
Any convex combination is a $C^*$-convex combination (using the coefficients $\sqrt t\,I$ and $\sqrt{1-t}\,I$),
so $C^*$-extremality implies
$\mathcal I_i = \operatorname{Ad}_{U_i}\circ \mathcal I$ for some unitaries $U_i\in\mathcal B(\mathcal H)$.

Fix $(A,a)\in \mathcal O(X)\times\mathcal A$ and set $T:=\mathcal I(A,a)$.
Then
\[
T=t\,U_1^*TU_1+(1-t)\,U_2^*TU_2,
\qquad \|T\|_2=\|U_i^*TU_i\|_2\ (i=1,2),
\]
where $\|\cdot\|_2$ is the Hilbert-Schimdt norm on the finite-dimensional Hilbert space $\bh.$ However, the sphere of a Hilbert
space contains no nontrivial line segments, we conclude
$U_1^*TU_1=T=U_2^*TU_2.$ 
Thus $\mathcal I_1=\mathcal I_2=\mathcal I$, proving extremality.
    \end{proof}

  We now recall a few notions from the theory of log-modular (respectively, factorization) algebras and nest algebras.  
The study of the factorization property for subalgebras of $C^*$-algebras is classical.  
A well-known example is Cholesky’s theorem, which establishes the factorization property for the algebra of upper-triangular matrices in $M_n(\mathbb C)$.  
For details and further references, see \cite{bhat_manish_Lattices_of_Logmodular_algebra}.  
These notions will play an important role in our characterization of $C^*$-extreme quantum instruments.  

\begin{dfn}
Let $\mathcal M$ be a subalgebra of a $C^*$-algebra $\mathcal B$.  
We say that $\mathcal M$ is \emph{log-modular} in $\mathcal B$ if
\[
\{a^*a : a \in \mathcal M,\; a^{-1}\in \mathcal M\}
\]
is norm dense in $\mathcal B_{+}^{-1}$, the set of all positive invertible elements of $\mathcal B$.  

In particular, if
\[
\mathcal B_{+}^{-1} = \{a^*a : a \in \mathcal M,\; a^{-1}\in \mathcal M\},
\]
then $\mathcal M$ is said to have the \emph{factorization property} in $\mathcal B.$
The factorization property is also referred to as \emph{strong log-modularity}.
\end{dfn}
\begin{rmrk}
    It is immediate from the definition that if a subalgebra $\mathcal M$ of a $C^*$-algebra $\mathcal B$ has the factorization property in $\mathcal B$, then $\mathcal M$ is automatically log-modular in $\mathcal B$.
\end{rmrk}
We now restate a weaker form of Corollary \ref{factorizationcorollary} using the language of factorization.

\begin{ppsn}\label{factorization proposition}
Let $\mathcal I:\mathcal O(X)\to CP(\mathcal A,\mathcal H)$ be a $C^*$-extreme instrument with minimal bi-dilation $(\mathcal K,\pi,E,V)$.  
Then the subalgebra
\[
\clm := \{\, S \in \{\pi(\mathcal A)E(\mathcal O(X))\}' : SVV^* = VV^*SVV^* \,\}
\]
has the factorization property in $\{\pi(\mathcal A)E(\mathcal O(X))\}'$.
\end{ppsn}
    
For any collection $\{P_i\}_{i\in\Lambda}$ of projections in $\bh$, $\vee_{ i\in\Lambda}P_i$ denotes the projection onto the
	smallest subspace containing ranges of all $P_i'$s, and $\wedge_{i\in\Lambda}P_i$ denotes the projection onto the intersection of ranges of all $P_{i}'$s.
	\begin{dfn}[Lattice]
		A collection $\scre$ of projections in a von Neumann algebra $\clb$ is called a lattice if
		$P\wedge Q$ and $P\vee Q\in\scre$ whenever $P,Q\in\scre.$
	\end{dfn}
	
	Let $\clm$ be a sub-algebra of a von Neumann algebra $\clb.$ Let $Lat_{\clb}\clm$ denote the lattice of all
	projections in $\clb$ whose ranges are invariant under every element of $\clm$ i.e.
	$Lat_{\clb}\clm =\{P\in\clb;P=P^2=P^*~\text{and}~ AP=PAP~\forall~A\in M\}.$
	If $\clb=\bh,$ we denote $Lat_{\clb}\clm$ simply by $Lat\clm.$ Interesting to note that if $\clm$ is also considered as a
	subalgebra of $\bh$ (where $\clb\subseteq \bh$), then we have
	$Lat_{\clb}\clm =\clb\cap Lat\clm.$ We note that $0,1 \in Lat_{\clb}\clm$ and $Lat_{\clb}\clm$ is closed under the operations $\vee$ and $\wedge$ of arbitrary
		sub-collection, as well as closed under weak operator topology (WOT).

	Dually, let $\scre$ be a collection of projections in $\clb$ (which may not be a lattice), and let $Alg_{\clb}\scre$
	(or $Alg\scre$ when $\clb =\bh$) denote the algebra of all operators in $\clb$ which leave range of every
	projection of $\scre$ invariant i.e.
	$Alg_{\clb}\scre =\{X\in\clb;XP =PXP~\forall~P\in\scre\}.$
	Again we note that
	$Alg_\clb\scre  =\clb\cap Alg\scre.$
	Also it is clear that $Alg_{\clb}\scre$ is a unital subalgebra of $\clb$, which is closed in WOT.
	\begin{dfn}[Nest]
		Let $\scre$ be a lattice of projections in a von Neumann algebra $\clb.$ Then the
		lattice $\scre$ is called
		a nest if $\scre$ is totally ordered by usual operator ordering i.e. for any $P,Q\in\scre,$ either $P\leq Q$
		or $Q\leq P$ holds true. $\scre$ is said to be a complete nest if $0,1 \in\scre$ and $\vee_{i\in\Lambda}P_i$ and $ \wedge_{i\in\Lambda}P_i~\in\scre$ for any arbitrary family $\{P_i\}_{i\in\Lambda}$ in $\scre.$
		
	\end{dfn}
	\begin{dfn}[Atom and Atomic nest]\label{atom}
	   Let $\mathscr E$ be a complete nest of projections in a von Neumann algebra $\mathcal B$.  
A non-zero projection $r\in\mathcal B$ is called an \emph{atom} of $\mathscr E$ if
\[
r = P - \vee_{Q<P} Q,
\]
for some $P \in \mathscr E$.   A complete nest $\mathscr E$ is said to be \emph{atomic} if there exists a countable collection of atoms $\{r_n\}$ of $\mathscr E$ such that
$\sum_n r_n = 1_{\mathcal B}$ in WOT.

	\end{dfn}
	\begin{dfn}[Nest sub-algebra]
		Let $\clm$ be a sub-algebra of a von Neumann algebra $\clb.$ Then $\clm$ is called
		a nest sub-algebra of $\clb$ (or nest algebra when $\clb = \bh$) if $\clm = Alg_{\clb}\scre$ for a nest $\scre$ in $\clb$ and is called $\clb$-reflexive (or reflexive when $\clb =\bh$ ) if $\clm = Alg_{\clb} Lat_{\clb}\clm.$
		
	\end{dfn}

    Next, we recall a special case of Corollary~4.8 from \cite{bhat_manish_Lattices_of_Logmodular_algebra}, presented here as Lemma~\ref{factorization lemma}, which characterizes log-modular algebras within finite-dimensional von Neumann algebras.
This lemma plays a crucial role in establishing a characterization of $C^*$-extreme instruments in finite dimensions.

	\begin{lmma}\label{factorization lemma}
		Let $\clb$ be a finite-dimensional von Neumann algebra. Then $\cln$ is a log-modular sub-algebra of $\clb$  if and only if it's a nest sub-algebra in $\clb.$ Moreover, any such $\cln$ is $\clb$-reflexive.
	\end{lmma}
	
	\begin{rmrk}\label{factorization lemma remark}
	   It is worth noting that in the preceding lemma, if $\mathcal N=\operatorname{Alg}_{\mathcal B}(\mathscr E)$ for some nest $\mathscr E$ in $\mathcal B$, then $\overline{\mathscr E}\subseteq \operatorname{Lat}_{\mathcal B}(\mathcal N)$, which is a finite, complete, atomic nest.
	\end{rmrk}

	\begin{thm}\label{direct sum decomposition of \cst extreme instruments}
		Let $\cli:\ox \to CP(\cla,\bh)$ be a  $C^*$-extreme UCP instrument,where $\mathcal{H}$ is a finite-dimensional Hilbert space. Then $\cli$ is unitarily equivalent to a direct sum of pure UCP instruments.
	\end{thm}
	\begin{proof}
		Let $(\clk,\pi,E,V)$ be the minimal bi-dilation tuple of the $C^*$-extreme UCP instrument $\cli.$ Since $\mathcal{H}$ is finite-dimensional, by  Theorem \ref{cstarextremeimpliesextreme} it follows that $\cli$ is an extreme UCP instrument. Furthermore, by Theorem~\ref{thm:extrme point criterion for instruments}, the map $D\mapsto V^*DV$ from $\{\pi(\cla)E(\clo(X))\}'$ to $\bh$ is injective.  Consequently,
         $\{\pi(\cla)E(\clo(X))\}'$
        is finite-dimensional. Now, since $\cli$ is $C^*$-extreme, Proposition \ref{factorization proposition} implies that the sub-algebra $$\clm:=\{T:TV(H)\subseteq V(H)~\&~ T\in\{\pi(\cla)E(\clo(X))\}'\},$$ has factorization in $\{\pi(\cla)E(\clo(X))\}'.$ By applying Lemma~\ref{factorization lemma} together with Remark~\ref{factorization lemma remark}, we conclude that $\clm$ is a reflexive sub-algebra of $\{\pi(\cla)E(\clo(X))\}'$ i.e. $$\clm=Alg_{\{\pi(\cla)E(\clo(X))\}'}Lat_{\{\pi(\cla)E(\clo(X))\}'}\clm$$ and that $Lat_{\{\pi(\cla)E(\clo(X))\}'}\clm$ contains a finite, complete, atomic nest. Let $\mathscr{E}$ denote such a finite, complete, atomic nest contained in $Lat_{\{\pi(\cla)E(\clo(X))\}'}\clm,$ and let $\{P_i:i=1,\cdots,n\}$ be the set of atoms of $\mathscr{E}.$ By definition \ref{atom}, each $P_i$ is a non-zero projection of the form $P-{P}_{-}$ for some $P\in\mathscr{E},$ where ${P}_{-}=\vee _{\{Q\in\mathscr{E};Q<P\}}Q.$ 
		 It is immediate to note that each $P_i,$ as well as any sub-projection of it, belongs to $$Alg_{\{\pi(\cla)E(\clo(X))\}'}\scre=Alg_{\{\pi(\cla)E(\clo(X))\}'}Lat_{\{\pi(\cla)E(\clo(X))\}'}\clm=\clm.$$
		Since $\{\pi(\cla)E(\clo(X))\}'$ is finite dimensional, each projection $P_{i},$ can be further decomposed into a finite set of minimal orthogonal sub-projections $P_{ij},$  i.e., $P_i=\sum_{j=1}^{n_i}P_{ij}$ with each $P_{ij}\in\clm.$  This collection $\{P_{ij}\}$ forms a resolution of the identity on $\clk,$ i.e.,  $\clk=\oplus_{i,j}\clk_{ij},$ where $\clk_{ij}= \text{range}(P_{ij})$ are mutually orthogonal sub-spaces.	This induces a direct sum decomposition of the spectral instrument $\pi E:$ $$\pi E=\oplus_{i,j}\pi_{ij}E_{ij},~\text{where}~ \pi_{ij}E_{ij}:=P_{ij}\pi E.$$ 
		The minimality of each $P_{ij}$ ensures that the  instruments $\pi_{ij}E_{ij}$ are irreducible. Since each $P_{ij}\in \clm,$ we also obtain a decomposition of the subspace $V(\mathcal{H}):$ as $V(\mathcal{H})=\oplus_{i,j}V_{ij}(\mathcal{H}),$ leading to a decomposition of $\mathcal{H},$ itself: $\clh=\oplus_{ij}\mathcal{H}_{ij},~\text{with}~V_{ij}:\mathcal{H}_{ij}\to\clk~\text{isometry}.$ Therefore  the instrument $\cli$ decomposes a finite direct sum of irreducible instruments: $\cli=\oplus_{ij} V_{ij}^*\pi_{ij}E_{ij}V_{ij}.$
	\end{proof}
	
	Combining Theorem \ref{direct sum of pure instruments}, Proposition \ref{mutually disjoint directsum}, and the earlier Theorem \ref{direct sum decomposition of \cst extreme instruments}, we obtain the following complete characterization of $\cst$-extreme instruments in finite dimensions.

	\begin{thm}\label{characterization of $\cst$-extreme instruments}
		Let $\cli:\ox \to CP(\cla,\bh) $ be a UCP instrument, where $\clh$ is finite-dimensional Hilbert space. Then $\cli$ is $C^*$-extreme if and only if there exist finitely many mutually disjoint irreducible instruments $\{\pi_iE_i\}^m_{i=1}$ on $X$ and nested sequence of compression $\cli^i_j (1\leq j\leq n_i)$ of each irreducible instruments $\pi_iE_i$ such that $\cli$ is unitarily equivalent to $\oplus_{i=1}^m\oplus_{j=1}^{n_i}\cli^i_j$.
	\end{thm}
	\begin{rmrk}
		This theorem recovers the characterizations of $\cst$-extreme unital completely positive (UCP) maps on finite-dimensional spaces established by Farenick and Zhou in \cite{farenickpams}. It also yields an analogous characterization for normalized positive operator-valued measures (POVMs).
	\end{rmrk}

We now point out an immediate consequence of the structural description of $\cst$-extreme instruments obtained in Theorem~\ref{characterization of $\cst$-extreme instruments}. In the finite-dimensional setting, this structural rigidity implies that the two natural dilations—CP sub-minimal and minimal  bi dilation must agree.
	
	\begin{crlre}\label{coro: \cst-extremity implies that the CP sub-minimal dilation matches with the bi-dilation}
		Let $\mathcal{I} : \ox \to \mathrm{CP}(\mathcal{A}, \mathcal{B}(\mathcal{H}))$ be a $\mathrm{C}^*$-extreme UCP instrument. If $\mathcal{H}$ is finite-dimensional, then the CP sub-minimal dilation coincides with the bi-dilation of the instrument.
	\end{crlre}
    However, it is interesting to note that the POVM subminimal dilation does not, in general, provide any information about the minimal bi-dilation.  
For instance, consider the instrument 
\[
\mathcal I:\mathcal O(\{1,2\})\times M_4(\mathbb C)\longrightarrow M_2(\mathbb C), 
\qquad 
\mathcal I(1,X)=V^*XV,\;\; \mathcal I(2,X)=0,
\]
where $V\in M_{4,2}(\mathbb C)$ is an isometry.  
This defines a pure instrument with a spectral POVM marginal, and hence $\mathcal I$ is clearly a $C^*$-extreme UCP instrument.  
Nevertheless, the POVM subminimal dilation is 
$(\mathbb C^2,\phi_{\mathcal I},\mu_{\mathcal I},V)$, 
whereas the minimal bi-dilation is 
$(\mathbb C^4,\pi,E,V)$ and the two are not same.

	\section{Instrument and Its Marginals Through Different Notions of Convexity}\label{Instrument and Its marginals through different notions of convexity} 
    In this section, we explore the relationship between an instrument and its marginals through the lens of classical and $\cst$-convexity structures of CP instruments introduced in Section \ref{Extremality and $\cst$-Convexity of Instruments}.
	\subsection{Extreme instruments and their marginals}
	
	In general the extremity of an instrument does not imply the extremity of its marginals. That is, there exist extreme instruments whose marginals—both the POVM and the CP part—are not extreme. This phenomenon was first discussed in \cite{Ariano}, and is illustrated in Example \ref{exmp: extreme instrument with non-extreme marginals}. As previously noted, the instrument $\cli$ in that example is an extreme instrument. However, the choice of the POVM $\mu$ clearly indicates that the POVM marginal is not extreme. Moreover, the non-extremity of the CP marginal follows from the commutativity of $\mu$, together with Theorem \ref{thm:extrme point criterion for instruments} for CP maps.

	Interestingly, we see below that under certain natural additional conditions, the extremality of the POVM marginal can be inferred from the extremality of the instrument itself. The same can't be said about the CP marginal. For instance, if an extreme instrument is commutative, then its POVM marginal must be a spectral measure (see Theorem \ref{thm : POVM marginal of extreme instruments with commutative range is spectral}). However, there exist examples (see Example \ref{exmp : $\cst$-extreme instrument wth non extreme CP marginal}) where the corresponding CP marginal is not extreme.
    
	\begin{thm}\label{thm : POVM marginal of extreme instruments with commutative range is spectral}
		Let $\cli:\ox\to CP(\cla,\bh)$ be an extreme UCP instrument in the set $I_{\clh}(X,\cla)$ with commutative range. Then the POVM marginal, $\mu_{\cli}:\ox\to\bh$ of $\cli$ is a spectral measure.
	\end{thm} 
	\begin{proof}
		Fix $A\in\ox$ such that $\mu(A)\neq0.$
		Define two POVMs $$\nu_1(B)=\mu(A\cap B)\mu(A^\complement) ~\text{and}~\nu_2(B)=\mu(A^\complement \cap B)\mu(A),~\forall~B\in\ox.$$ It is easy to verify that $\nu_i\leq\mu$ for $i=1,2$ and further, $\nu_1(X)=\mu(A)\mu(A^\complement)=\nu_2(X).$ Since $\mu$ is an extreme point, it follows from the previous theorem that $\nu_1=\nu_2.$ In particular, evaluating at $B=A$ we obtain,  $\nu_1(A)=\nu_2(A)=\mu(A)\mu(A^\complement)=0.$
		Now using the relation $\mu(A)+\mu(A^\complement)=1,$ we conclude that $\mu(A)$ is a projection. Since $A$ was arbitrary, it follows that all $\mu(A)$'s are projections. Hence, $\mu$ is a spectral measure.
	\end{proof}
In the converse direction, it can be seen that   the extremity of both marginals implies the extremity of the instrument. This result is due to E. Haapasalo, T. Heinosaari, and J.-P. Pellonpää (see Theorem 4.1 in \cite{pellonpaapieces}). Although the original result was established for CP maps on  tensor products of von Neumann algebras, the underlying technique adapts naturally to our framework. For the reader’s convenience,  we restate the result here in our setting.

	\begin{thm}
		Let $\cli:\ox\to CP(\cla,\bh)$ be a UCP instrument. If both the marginals $\phi_\cli$ and $\mu_\cli$ are extreme, then the instrument $\cli$ is extreme.
	\end{thm}

Moreover, in the same theorem, it was shown that the extremity of a single marginal is sufficient to identify the instrument uniquely. 
Nevertheless, the extremity of the marginals alone does not guarantee that the instrument is decomposable. This is a feature that sharply contrasts with the classical case. 
This distinction is highlighted in the following example, which uses the correspondence between regular POVMs on a compact Hausdorff space $X$ and completely positive maps on $\cx$ (see Chapter~4, \cite{paulsen_book}).
\begin{xmpl}\label{extremity of marginals doesn't gaurantee decomposability}
Let $X = \{1,2,3,4\},~\ox= \mathcal{P}(X)$ and $\clh = \mathbb{C}^2$. Set $\omega = e^{2\pi i/3}$ and  define $\mu :\ox \to \mathcal{B}(\clh)$ by,
\[
\mu(\{1\}) = \frac12 \begin{bmatrix} 1 & 0 \\ 0 & 0 \end{bmatrix}, \quad
\mu(\{2\}) = \frac16 \begin{bmatrix} 1 & \sqrt{2} \\ \sqrt{2} & 2 \end{bmatrix},
\]
\[
\mu(\{3\}) = \frac16 \begin{bmatrix} 1 & \sqrt{2}\omega^2 \\ \sqrt{2}\omega & 2 \end{bmatrix}, \quad
\mu(\{4\}) = \frac16 \begin{bmatrix} 1 & \sqrt{2}\omega \\ \sqrt{2}\omega^2 & 2 \end{bmatrix}.
\]
Then $\mu$ is an extreme POVM. Let $(\clk, E, V)$ be its minimal Naimark dilation.  
Consider the instrument $\cli : \ox \to CP(\cx, \mathcal{B}(\clh))$ given by
\[
\cli(A)(a) = V^* \pi_{\mu}(a) E(A) V,
\]
where $\pi_\mu : \cx \to \mathcal{B}(\clk)$ is the $*$-homomorphism corresponding to $E$ defined by $$E(a)= \sum _{i=1}^4a(i)E(\{i\}), ~~a\in C(X).$$  The extremality of $\mu$ ensures that $\cli$ is an extreme UCP instrument. However, $\cli$ is not decomposable.
\end{xmpl}

	\subsection{\texorpdfstring{$\cst$}{C*}-extreme instruments and their marginals}
	In this section, we investigate the interplay between $\cst$-convexity properties of instruments and their marginals. 
    
    Example \ref{exmp: extreme instrument with non-extreme marginals} revealed that, even in finite dimensions, the classical extremality of an instrument does not, in general, imply the extremality of its marginals.
In contrast to the classical convex setting, the framework of $\cst$-extremity reveals a more intricate relationship. We will see in Theorem \ref{thm: POVM marginal of $C^*$-extreme instruments are spectral in f.d.} below that the $\cst$-extremity of an instrument does ensure that its POVM marginal is spectral -
	hence, $\cst$-extreme at least in the finite-dimensional setting. However, no such conclusion can be drawn for the CP marginal. This asymmetry is illustrated in the following example and in Theorem \ref{thm: POVM marginal of $\cst$-extreme instruments are spectral in f.d.}.

	\begin{xmpl}\label{exmp : $\cst$-extreme instrument wth non extreme CP marginal} Consider the instrument $\cli:\ox\to CP(M_2(\bbc),M_2(\bbc)),$ on the set $X=\{1,2\},$ defined by
		$\cli(i)(A)=a_{ii}E_{ii}, ~i=1,2,$
		where $A=\sum _{i,j=1}^2a_{ij}E_{ij}$, and $E_{ij},i,j=1,2,$ are the standard matrix units in $M_2(\bbc),$.
	 It is easy to verify that $\cli$ is a $\cst$-extreme instrument with commutative range. However, the associated CP marginal $\phi_{\cli}:M_2(\bbc)\to M_2(\bbc)$, mapping a matrix to its diagonal part,
	is not extreme in the convex set of unital completely positive maps, and hence not $\cst$-extreme. \end{xmpl}
	
	We now use the explicit decomposition of $C^*$-extreme instruments obtained in Theorem \ref{characterization of $\cst$-extreme instruments} to show that $\cst$-extremity of an instrument implies the $\cst$-extremity of its POVM marginal in the finite-dimensional setting.

\begin{thm}\label{thm: POVM marginal of $\cst$-extreme instruments are spectral in f.d.}
Let $\cli : \ox \to CP(\cla, \bh)$ be a $\cst$-extreme UCP instrument with $\clh$  finite-dimensional. Then its POVM marginal $\mu_\cli$ is $\cst$-extreme.
\end{thm}
	\begin{proof}
		By Theorem \ref{characterization of $\cst$-extreme instruments}, any $\cst$-extreme instrument $\cli$ admits a decomposition of the form: $$\cli=\oplus_{i=1}^m\oplus_{j=1}^{n_i}\cli^i_j,$$ where $\cli^i_j (1\leq j\leq n_i)$ is a nested sequence of compressions of irreducible instruments $\pi_iE_i, 1\leq i \leq m .$ By Remark \ref{rmk: POVM marginal of compression pure instrument is trivial}, the POVM marginal of each $\cli^i_j$ is trivial. Since the POVM marginal of $\cli$ is the direct sum of the marginals of the $\cli^i_j,$  it follows that the POVM  marginal of $\cli$ is a direct sum of trivial measures, hence spectral, and therefore $\cst$-extreme. 	\end{proof}

	In parallel with the case of extremality (see Theorem \ref{thm : POVM marginal of extreme instruments with commutative range is spectral}), we will see in Corollary \ref{cor : POVM marginal of a $C^*$-extreme instrument with a commutative range is spectral} that commutativity of a $\cst$-extreme instrument guarantees that its POVM marginal is spectral. This result, however, appears as a consequence of the following more general theorem, which captures a broader structural phenomenon. The formulation and the underlying idea of this theorem are motivated by the main result (Theorem 3.8) of \cite{manishcmp} and so we omit the proof.

	\begin{thm}\label{m}
		Let $\cli$ be a  $C^*$-extreme point in $I_\clh(X,\cla)$. If $E\in\ox$ satisfies $\cli(A,a)\cli(E,1_\cla)=\cli(E,1_\cla)\cli(A,a)$ for all $A\subseteq E$  in $\ox,$ and $a\in\cla$, then $\cli(E,1_\cla)$ is a projection. In particular, if $\cli(E,1_\cla)$ commutes with all $\cli(B,a)$ for $B\in \ox$ and $a\in\cla,$ then $\cli(E,1_\cla)$ is a projection.
	\end{thm} 
	The following two corollaries are immediate:
	\begin{crlre}
		For every atomic $\cst$-extreme instrument, the associated POVM marginal is a spectral measure.
	\end{crlre}
	\begin{crlre}\label{cor : POVM marginal of a $\cst$-extreme instrument with a commutative range is spectral}
	The associated POVM of every commutative $\cst$-extreme instrument is a spectral measure.
	\end{crlre}

	It was established in Theorem 4.1, \cite{pellonpaapieces} 
    that if both the marginals of an instrument are extreme, then the instrument itself must also be extreme. It is natural to ask whether an analogous result holds in the setting of $\cst$-convexity. In this section, we provide an affirmative answer to this question in the finite-dimensional case.
	
	Before presenting the main result, we establish a few preparatory results that will play a crucial role in the proof. To this end, let us recall the relationship between sub-minimal dilations and minimal bi-dilations as discussed in Section~\ref{Relation between Sub-minimal dilation and minimal Bi-dilation}. Let $\cli: \ox \to CP(\cla, \bh)$ be a UCP instrument, and suppose that $(\clk, \pi, E, V)$ is its minimal bi-dilation. Let $P_1$  be the orthogonal projection onto the subspace $\clk_1 := [\pi(\cla)V(\clh)] \subseteq \clk$. Then the minimal Stinespring dilation of the associated CP marginal $\phi_\cli$ is given by the triple $(\clk_1, P_1 \pi P_1, V)$. Consequently, the CP sub-minimal dilation of the instrument $\cli$ is described by the quadruple $(\clk_1, P_1 \pi P_1, P_1 E P_1, V)$.
	
	\begin{thm}\label{thm : equivalence of bi and sub-minimal dilation}
		Let $\cli: \ox \to CP(\cla, \bh)$ be a UCP instrument with minimal bi-dilation quadruple $(\clk, \pi, E, V)$, and let $P_1$ denote the orthogonal projection onto the subspace $\clk_1:= [\pi(\cla)V(\clh)]\\ \subseteq \clk$. Suppose that the associated UCP map $\phi_\cli$ is extreme and the associated POVM marginal $\mu_\cli$ is spectral. Then the compression $P_1 E P_1$ is a spectral measure. Consequently, the CP sub-minimal dilation coincides with the minimal bi-dilation of $\cli$.

	\end{thm}
	
	\begin{proof}
		Since the POVM marginal $\mu_\cli$ is spectral, Corollary~\ref{decomposable-sufficientcondition} together with Theorem~\ref{decomposibleequivalent-1} implies that
\begin{eqnarray}\label{decomposibleeqn1}
P_1E(A)P_1VV^*=VV^*P_1E(A)P_1VV^*=VV^*P_1E(A)P_1, \text{for all}~ A\in\ox,
		\end{eqnarray}
		Recall that $\mu_\cli = V^* P_1 E P_1 V$. The spectrality of $\mu_\cli$, in conjunction with \eqref{decomposibleeqn1}, leads to the equivalence:
		\begin{eqnarray*}			\mu_\cli(A)=\mu_\cli(A)^2&\iff&
			V^*P_1 E(A) P_1 V=	V^*P_1 E(A) P_1 V	V^*P_1 E(A) P_1 V\\&\iff&
			V^*P_1 E(A) P_1 V= V^*(P_1 E(A) P_1)^2 V,~\forall~A\in\clo(X).
		\end{eqnarray*}
		Since, the associated UCP map $\phi_\cli$ is extreme, it follows that the map: $D\mapsto V^*P_1DP_1V$, is injective on the von Neumann algebra $\{ P_1 \pi(\cla) P_1 \}'$. Therefore, the equality above implies that $P_1 E(A) P_1$ is a projection for each $A \in \ox$, i.e., the POVM $P_1 E P_1$ is spectral. Thus, the quadruple $(\clk_1, P_1 \pi P_1, P_1 E P_1, V)$ is a bi-dilation of the instrument $\cli$. Since this is already a sub-minimal dilation, it follows that it is also minimal as a bi-dilation.
	\end{proof}
	Having established  the preparatory results, we are in a position to present some principal contributions of this paper.
	\begin{thm}\label{marginalconvexity}

Let $\mathcal I:\mathcal O(X) \to CP(\mathcal A, \mathcal B(\mathcal H))$ be a UCP instrument such that its POVM marginal $\mu_{\mathcal I}$ is spectral and its CP marginal admits the decomposition
\[
\phi_{\mathcal I} = \oplus_{i=1}^{\ell}\, \psi_i \otimes 1_{\mathbb{C}^{n_i}},
\]
where $\psi_i:\mathcal A \to \mathcal B(\mathcal H_i)$, $i=1,\dots,\ell$, are pure unital completely positive maps with $\dim \mathcal H_i < \infty$.  
Assume that each $\psi_i$ admits a minimal Stinespring dilation $(\mathcal H_\pi,\pi,V_i)$ with respect to a fixed irreducible representation $\pi$ of $\mathcal A$, and that the family of subspaces $\{\operatorname{range}(V_i)\}_{i=1}^\ell$ forms a nest in $\mathcal H_\pi$. Then $\mathcal I$ is $C^*$-extreme.
	
	\end{thm}
	\begin{proof}
		We begin by establishing the necessary notations. Let
		$\clk=\oplus_{i=1}^{\ell}(\clh_\pi\otimes \bbc^{n_i}),$ and consider the representation 
		$ \rho: \cla \to \clb(\clk)$,    given by $$ \rho(a)=\oplus_{i=1}^{\ell}(\pi(a)\otimes 1_{\bbc^{n_i}}),~~\forall a \in \cla,$$
		where $\pi$ is a fixed irreducible representation of $\cla.$
		Define the isometry, $V:\clh\to\clk $ by $V=\oplus_{i=1}^{\ell}(V_i \otimes 1_{\bbc^{n_i}}).$ It is evident that the triple $(\clk,\rho,  V)$ forms a minimal Stinespring dilation for the UCP map $\phi_\cli$. By Theorem \ref{characterization of $\cst$-extreme instruments} for UCP maps we can conclude that $\phi_\cli$
		is $\cst$-extreme. Since $\clh$ is finite-dimensional and $\phi_\cli$ is $C^*$-extreme, it is also classically extreme. Since $\cli$ satisfies the hypothesis of Theorem \ref{thm : equivalence of bi and sub-minimal dilation}, we conclude that the CP sub-minimal dilation of $\cli$ coincides with its minimal Naimark dilation, which is given by the quadruple $(\rho, E, V, \clk)$, where $E:\clo(X) \to \rho(\cla)'\subset \clb(\clk)$
		is a spectral measure.
		Since $\pi$ is irreducible, we have $\pi(\cla)' = \bbc \cdot 1_{\clh_\pi}$, and therefore,
		\[\rho(\cla)^ \prime=\{1_{\clh_\pi}\otimes A: A \in M_m(\bbc), ~~m=\sum_{i=1}^{\ell}n_i\}.\]
		As $E$ is a spectral measure and takes values in $\rho(\cla)'$, with out loss of generality it decomposes as $E=\oplus_{i=1}^m E_i,$ where each $E_i$ is a trivial spectral measure on $X$ with values in $\clb(\clh_\pi)$. Accordingly, the instrument $\cli$ can be expressed as, \[\cli=\oplus_{i=1}^mW_i^*\pi_iE_iW_i,\] where each $\pi_i = \pi$, and the isometries $W_i$ are drawn from $\{ V_1, \dots, V_l \}$ with the common dilation space $\clk$. It follows immediately that the instruments $\{ \pi_i E_i \}_{i=1}^m$ are irreducible. Therefore, by invoking Theorem \ref{characterization of $\cst$-extreme instruments}, we conclude that $\cli$ is indeed $C^*$-extreme.

	\end{proof}

	Combining Proposition \ref{mutually disjoint directsum} with Theorem \ref{marginalconvexity}, we conclude that in finite dimensions the $C^*$-convexity of the marginals of an instrument ensures that the instrument itself is $C^*$-extreme. Moreover, under additional assumptions, the converse implications can also be established.

\begin{thm}\label{Final}
Let $\cli : \ox \to CP(\cla,\bh)$ be a UCP instrument, where $\mathcal{H}$ is finite dimensional.  
Then the marginals $\mu_\cli$ and $\phi_\cli$ are $\cst$-extreme if and only if:
\begin{enumerate}
    \item $\cli$ is $\cst$-extreme; and
    \item if $\cli_1, \cli_2 : \ox \to CP(\cla,\bh)$ dominated by $\cli$ are two pure instruments such that  $\phi_{\cli_1}$ and $\phi_{\cli_2}$ dilate to a common  irreducible representation then one of them  is a compression of the other.
\end{enumerate}
\end{thm}

\begin{proof}
\noindent
$(\implies)$ In the forward direction, the $\cst$-extremity of $\cli$ follows from Proposition \ref{mutually disjoint directsum} along with the  Theorem \ref{marginalconvexity}.  
To prove (2), let $\clj' : \ox \to CP(\cla,\bh)$ be a pure instrument with minimal bi-dilation $(\clk_{\clj'},\pi_{\clj'},E_{\clj'},V_{\clj'})$, and let $\clj : \ox \to CP(\cla,\bh)$ be the direct sum of pure instruments such that $\clj' \leq \clj$, where
$\clj = \oplus_{i=1}^n V_i^* \pi_i E_i V_i.
$
Then necessarily $\clj' = t V_\ell^* \pi_\ell E_\ell V_\ell$ for some $\ell \in \{1,\dots,n\}$ and $t \in [0,1]$.  

Since $\phi_\cli$ is $\cst$-extreme, by Theorem~\ref{characterization of $\cst$-extreme instruments} for UCP maps, there exists a family of disjoint irreducible representations $\{\pi_i\}_{i=1}^n$ together with a nested sequence of pure CP maps $\{V_{ij}^* \pi_i V_{ij}\}_{j=1}^{n_i}$ for each $i$, such that
\[
   \phi_\cli \;=\; \oplus_i \oplus_j V_{ij}^* \pi_i V_{ij}.
\]

Let $\cli_1, \cli_2 : \ox \to CP(\cla,\bh)$ be pure instruments dominated by $\cli$ such that $\phi_{\cli_1}$ and $\phi_{\cli_2}$ dilate to the same irreducible representation $\pi : \cla \to \bk$. Then $\phi_{\cli_i} \leq \phi_\cli$ for $i=1,2$. From the discussion in the beginning it follows that $\pi = \pi_m$ for some $m \in \{1,\dots,n\}$, and moreover
\[
  \phi_{\cli_1} = t_1 V_{m j_1}^* \pi_m V_{m j_1}, 
  \quad 
  \phi_{\cli_2} = t_2 V_{m j_2}^* \pi_m V_{m j_2},
  \quad t_i \in [0,1].
\]
Since the family $\{ V_{ij}^* V_{ij} \}$ is nested for each $i$, it follows that $\phi_{\cli_1}$ and $\phi_{\cli_2}$ are compressions of either one, establishing condition (2).  


$(\impliedby)$ For the converse, the $\cst$-extremity of the POVM marginal $\mu_\cli$ follows from Theorem~\ref{thm: POVM marginal of $\cst$-extreme instruments are spectral in f.d.}. Since $\cli$ itself is $\cst$-extreme, by Theorem~\ref{characterization of $\cst$-extreme instruments} we have  $\cli \;=\; \oplus_i \oplus_j V_{ij}^* \pi_i E_i V_{ij},$
where $\{\pi_i E_i\}_{i=1}^n$ are disjoint irreducible instruments, and $\{V_{ij}^* \pi_i E_i V_{ij}\}$ is a nested sequence of pure instruments for each $i$.  
Thus,  $\phi_\cli \;=\; \oplus_i \oplus_j V_{ij}^* \pi_i V_{ij}.$

To show that $\phi_\cli$ is $\cst$-extreme, it suffices to prove that
\[
  \phi_\cli \;=\; \oplus_{i'} \oplus_{j'} W_{i'j'}^* \pi_{i'} W_{i'j'},
\]
where $\{\pi_{i'}\}_{i'=1}^m \subseteq \{\pi_i\}_{i=1}^n$ are disjoint irreducible representations and for each $i'$, the family $\{W_{i'j'} W_{i'j'}^*\}$ is nested, with $\{W_{i'j'}\} \subseteq \{V_{ij}\}$.  
The disjointness of $\{\pi_{i'}\}$ follows directly, and condition (2) guarantees the nesting property. Hence, $\phi_\cli$ is $\cst$-extreme. \end{proof}

\section*{Acknowledgments}
Bhat gratefully acknowledges funding from  ANRF  (India) through J C
Bose Fellowship No. JBR/2021/000024.	
	\bibliography{references}
	\bibliographystyle{plain}

\end{document}